%% file: preprintGKVW.tex
\documentclass[a4paper,final]{amsart}

\usepackage[foot]{amsaddr} 

\usepackage[utf8]{inputenc}
\usepackage[T1]{fontenc}
\usepackage{amsmath,amssymb,amsfonts}
\usepackage{enumerate,color,graphicx,algorithm}
\usepackage{cite}
\usepackage{booktabs,dcolumn}
\usepackage{mathtools}
\usepackage{enumitem}
\usepackage{xspace}

\usepackage{xcolor}
\usepackage{todonotes}
\usepackage{changes}
\usepackage[left]{showlabels}

\numberwithin{equation}{section}

\theoremstyle{plain}
\newtheorem{thm}{Theorem}[section]
\newtheorem{lem}[thm]{Lemma}
\newtheorem{cor}[thm]{Corollary}

\theoremstyle{remark}
\newtheorem{rem}[thm]{Remark}
\newtheorem{ex}[thm]{Example}

\newcommand{\defas}{:=}
\newcommand{\dual}[2]{\left\langle{#1},{#2}\right\rangle}

\DeclareMathOperator*{\infimum}{inf\vphantom{p}}
\newcommand{\norm}[1]{\left\|{#1}\right\|}
\newcommand{\snorm}[1]{\left|{#1}\right|}
\newcommand{\sprod}[2]{\left({#1},{#2}\right)}
\newcommand{\ud}{u_d}
\newcommand{\vphi}{\varphi}

\newcommand{\N}{{\mathbb{N}}}
\newcommand{\R}{{\mathbb{R}}}

\newcommand{\weak}{\rightharpoonup}

\begin{document}

\title[Quasi-best approximation in PDE constraint optimization]{Quasi-best approximation 
in optimization \\ with PDE constraints}

\begin{abstract}
We consider finite element solutions to quadratic optimization problems, where 
the state depends on the control via a well-posed linear partial differential 
equation. Exploiting the structure of a suitably reduced optimality system, we prove that the combined error in the state and adjoint state of the variational discretization is bounded by the best approximation error in the underlying discrete spaces. The constant in this bound depends on the inverse square-root of the Tikhonov regularization parameter. Furthermore, if the operators of control-action and observation are compact, this quasi-best-approximation constant becomes independent of the Tikhonov parameter as the meshsize tends to $0$ and we give quantitative relationships between meshsize and Tikhonov parameter ensuring this independence. We also derive generalizations of these results when the control variable is discretized or when it is taken from a convex set.
\end{abstract}

\author[F. Gaspoz]{Fernando Gaspoz}
\address[Fernando Gaspoz]{Technische Universit\"at Dortmund,
        Fakult\"at f\"ur Mathematik, Vogelpothsweg 87,
        44227 Dortmund, Germany.}
\email{fernando.gaspoz@tu-dortmund.de}
\author[C. Kreuzer]{Christian Kreuzer}
\address[Christian Kreuzer]{Technische Universit\"at Dortmund,
        Fakult\"at f\"ur Mathematik, Vogelpothsweg 87,
        44227 Dortmund, Germany.}
\email{christian.kreuzer@tu-dortmund.de}
\author[A. Veeser]{Andreas Veeser}
\address[Andreas Veeser]{Dipartimento di Matematica 'F. Enriques',
Università degli Studi di Milano,
Via C. Saldini, 50,
20133 Milano, Italy.}
\email{andreas.veeser@unimi.it}
\author[W. Wollner]{Winnifried Wollner}
\address[Winnifried Wollner]{Technische Universit\"at Darmstadt,
        Fachbereich Mathematik, Dolivostr. 15,
         64293 Darmstadt, Germany.}
\email{wollner@mathematik.tu-darmstadt.de}
\date{\today}
\maketitle

\input{content.tex}

\bibliographystyle{siam}
\bibliography{./lit}
\end{document}

%% file: content.tex
\section{Introduction}

Optimization problems with PDE constraints are ubiquitous.  A basic,
and regularly considered, example is
\begin{equation}
\label{sim-min}
 \min_{(q,u)\in L^2 \times H^1_0}
  \frac12 \snorm{ u-\ud }_0^2 + \frac\alpha2 \snorm{q}_0^2
\quad\text{subject to}\quad
 -\Delta u=q
\end{equation}
where $\snorm{\cdot}_0$ denotes the $L^2$-norm over some underlying domain, 
$\ud$ is the desired state and $\alpha>0$ scales the cost of the control.
Additionally, constraints on the control $q$ and/or the state $u$ can be
added, and the error due to a discretization of the state equation, and 
possibly the control, have been analyzed. For piecewise constant 
discretizations 
of the control this has been done in~\cite{Falk:1973,Geveci:1979}
including possible box-constraints on the control variable, see also 
the summary of obtainable 
convergence orders including Neumann-control in~\cite{Malanowski:1982}. 
The consideration of element wise linear functions for the 
control has been done in~\cite{CasasTroeltzsch:2003,Roesch:2006} in the 
presence of control constraints.

In~\cite{Hinze:2005} it was observed, that the minimization problem could be 
solved without prescribing a discretization of the control since the control 
can be recovered from the optimality condition and thus a discretization of 
the control is induced by the discretization for the state equation. With 
this $O(h^2)$ convergence 
for the control in $L^2$ could be shown even in the presence of box 
control-constraints. It was observed by~\cite{MeyerRoesch:2004} that the same 
convergence order can be obtained if a discretized control is used and a 
post-processing step based upon the optimality conditions is applied.

Due to the structure of the objective in~\eqref{sim-min} these above 
mentioned estimates make use of the `natural norm'
\[
 \snorm{u}_0 + \sqrt{\alpha}\snorm{q}_0.
\]
Although this norm is natural due to the functional, it induces a scaling 
$\sqrt{\alpha}$ in all estimates involving the control. Further estimates, 
for instance of $H^1$-norms 
of the state thereby also contain this scaling. Moreover, the above 
`natural norm' is not balanced in terms of approximation accuracy, i.e., 
the error of the state in $L^2$ will typically decay at least as fast as 
the error of the control.

The later effect, however, is invisible as long as the 
approximation accuracy of both terms is limited by the selected discrete spaces, 
and not by the regularity of the solutions, as it is typically the 
case for the model~\eqref{sim-min}. However, in the presence of 
pointwise constraints on the state, see,
e.g.,~\cite{CasasMateos:2002-1,DeckelnickHinze:2007b,Meyer:2006,DeckelnickHinze:2007,NeitzelWollner:2018}  
or the gradient of the
state~\cite{DeckelnickGuentherHinze:2007,GuentherHinze:2009,OrtnerWollner:2009,Wollner:2012}
optimal order estimates 
can only be obtained for the control variable; while numerics shows a
faster  
convergence of the error in the state variable in $L^2$.

As an alternative to the aforementioned works, one may combine the
error in the state with error in the (suitably rescaled) adjoint
state, measuring both in the norms that are given by the functional
analytic set-up of the PDE constraint. For problem
\eqref{sim-min}, this leads to the norm 
\begin{equation}
\label{err-norm}
 \norm{x}^2
 \defas
 \snorm{u}_1^2 + \frac{1}{\alpha} \snorm{p}_1^2,
 \quad
x = (u,z),
\end{equation}
where $\snorm{\cdot}_1$ denotes the $H^1_0$-norm. For respective
counterparts of \eqref{err-norm}, Chrysafinos and Karatzas
\cite{ChrysafinosKaratzas:2015,Chrysafinos.Karatzas:12} prove
so-called symmetric error estimates or quasi-best approximation
results. The growth of the quasi-best-approximation constant is
limited by $\alpha^{-2}$ and $\alpha^{-3/2}$, respectively.

In this article, we prove abstract quasi-best approximation results,
where the discretization error is measured in a counterpart of
\eqref{err-norm}. In order to illustrate our results, assume that the
underlying domain is convex, let $(V_h)_h$ be a sequence of conforming
finite-dimensional spaces that approximates $H^1_0$, and consider the
variational discretization of \eqref{sim-min}.  If we denote by
$x_h=(u_h,p_h)$ the pairs of approximate primal and dual states, our
results yield (cf.\ Theorem~\ref{T:qba} and Example~\ref{ex:sim-opt}) 
\begin{equation*}
 \norm{x-x_h}
 \leq
 \nu_h
 \inf_{v_h \in V_h \times V_h} \norm{x-v_h}
\end{equation*}
with
\begin{equation*}
\nu_h
\leq
\kappa_\alpha \defas 2 \left( 1 + C_F \left( 1 + \frac{2 C_F}{\sqrt{\alpha}} \right)\right)
\quad\text{and}\quad
|\nu_h - 1| \leq C_\mathcal{I} \kappa_\alpha h
\quad\text{as }h\to0.
\end{equation*}
Here $C_F$ is the constant in the Friedrichs inequality and $C_\mathcal{I}$ is an interpolation constant depending on the shape regularity on the underlying meshes. In contrast to the first, non-asymptotic relationship, the second,
asymptotic one exploits the compactness of the observation and
control-action operators and elliptic regularity theory. Notably, the
latter reveals that C\'ea's lemma, which holds for the constraint
discretization, is recovered as $h\to0$ and, in particular, ensures an
approximation quality independent of $\alpha$ for
$h=O(\sqrt{\alpha})$. 

The rest of the paper proceeds as follows. In Section~\ref{sec:2}, we state
precisely the considered problem class, allowing for any linear,
bounded, and inf-sup-stable operator in the constraint. Furthermore,
we reduce the optimality system by eliminating the control, and we lay
the groundwork for our results by a careful discussion of the
continuity and nondegeneracy properties of the associated bilinear
form. 

Section~\ref{sec:3} constitutes the core of this work and establishes
quasi-best approximation for the variational discretization. To this
end, the variational discretization is viewed as a Petrov-Galerkin
method and we employ the formula for the quasi-best-approximation
constant in Tantardini and Veeser \cite{Tantardini.Veeser:16}. For the
asymptotic behavior of the quasi-best-approximation constant, we
additionally invoke a duality argument, which is similar to, but
simpler than, Schatz \cite{Schatz:1974}. 

The last two sections center on generalizations of these results. In
Section~\ref{sec:4}, we consider approximate control-action operators,
covering in particular the discretization of the control
variable. Finally, Section~\ref{sec:5}, 
deals with nonlinear optimality systems arising from additional convex constraints
for the control. The derived results complement those of the linear
case and the simplification of Schatz' argument comes in quite
useful. 

\section{Model optimization problem and reduced optimality system}
\label{sec:2}
%
%
%

We introduce our model optimization problem.  Assume that the control
variable $q$ is taken from a real Hilbert space $Q$ with scalar
product  
$\sprod{\cdot}{\cdot}_Q$ and induced norm $\norm{\cdot}_Q$.  Its
corresponding state  $u\in V_1$ is determined by solving a linear
boundary value problem of the form 
\begin{equation}
\label{bvp}
 Au = Cq
\end{equation}
with the following setting:
\begin{itemize}
\item The \emph{state space} $V_1$ is a Hilbert space with scalar 
product $\sprod{\cdot}{\cdot}_{1}$ and induced norm
$\norm{\cdot}_{1}$. Its dual and the corresponding duality pairing are
indicated with $V_1^*$ and $\dual{\cdot}{\cdot}_{1}$, respectively. 
\item The \emph{differential operator} $A$ is induced by bilinear form
  $a\colon V_1 \times V_2 \to \R$, where $V_2$ is a second Hilbert
  space with scalar product $\sprod{\cdot}{\cdot}_{2}$, induced norm
  $\norm{\cdot}_{2}$, dual space $V_2^*$, and dual pairing
  $\dual{\cdot}{\cdot}_{2}$. We assume that the bilinear form $a$ is
  bounded and satisfies the following inf-sup conditions: 
\begin{subequations}
\label{ass-on-a}
\begin{gather}
\label{cont-a}
 M_a
 \defas
 \sup_{\norm{v_1}_1=1,\, \norm{v_2}_{2}=1} a(v_1,v_2) < \infty,
\\
\label{non-deg-a}
 \forall v_1 \in V_1
\quad
 \Big( \forall v_2\in V_1 \; a(v_1,v_2) = 0 \Big) \implies v_1 =0.
\\
\label{inf-sup-a}
 m_a
 \defas
 \infimum_{\norm{v_2}_2=1} \; \sup_{\norm{v_1}_{1}=1} a(v_1,v_2) > 0,
\end{gather} 
\end{subequations}
Employing well-known inf-sup theory (cf., e.g.,
Babu\v{s}ka~\cite{Babuska:71}), we see that the operator $A\colon V_1
\to V_2^*, v_1 \mapsto a(v_1,\cdot)$ is linear and boundedly
invertible. 
\item The \emph{control-action operator} $C\colon Q \to V_2^*$ is
  linear and bounded with constant $M_C$. 
\end{itemize}
Our goal is then to numerically solve the \emph{constrained optimization 
problem}
\begin{equation}
\label{mod-min}
 \min_{(q,u)\in Q \times V_1}
  \frac12\norm{Iu-\ud }_W^2 + \frac\alpha2 \norm{q}_Q^2
\quad\text{subject to}\quad
 Au = Cq
\end{equation}
where we assume in addition:
\begin{itemize}
\item The \emph{desired ``state''} $\ud$ is an element of a Hilbert space $W$ with scalar product $\sprod{\cdot}{\cdot}_W$ and induced norm $\norm{\cdot}_W$.
\item The \emph{observation operator} $I\colon V_1 \to W$ is linear, and bounded with constant $M_I$.
\item The \emph{cost of the control}, which can be viewed as a Tikhonov 
regularization, is scaled with the parameter $\alpha>0$.
\end{itemize}

Problem~\eqref{mod-min} is a quadratic minimization problem with a linear 
constraint.  The objective function is convex with respect to 
\begin{equation}
\label{enorm}
 (u,q) \mapsto ( \norm{Iu}_W^2+\alpha\norm{q}_Q^2 )^{1/2}
\end{equation}
and strictly convex in $q$.  Consequently, 
standard arguments ensure the existence of a unique solution; 
see, e.g.,~Lions~\cite[Theorem~1.1]{Lions:1971} 
or~Tr\"oltzsch~\cite[Chapter~2.5]{Troeltzsch:2005}.

If $Q=L^2=W$, $V_1=V_2=H^1_0$, $A=-\Delta$ is the (weak) Laplacian,
and $C$ and $I$ 
are the canonical compact immersions $L^2 \to \big( H^1_0 \big)^*$ and 
$H^1_0 \to L^2$, then~\eqref{mod-min} simplifies to the optimization 
problem~\eqref{sim-min} in the introduction.  Notice that, in this case, the operators 
$C$ and $I$ are related by $C^*=I$.

To formulate the \emph{optimality system} for~\eqref{mod-min}, it is useful to 
define the adjoint operators $A^*$, $C^*$, $I^*$ of $A$, $C$, $I$ 
by
\[
 A^* v_2 = a(\cdot,v_2),
\quad
 \sprod{q}{C^* v_2}_Q = \dual{Cq}{v_2},
\quad
 \dual{I^* w}{v_1}_1 = \sprod{Iv_1}{w}_W
\]
for all $v_1\in V_1$, $v_2\in V_2$, $q\in Q$, $w\in W$.  Thanks to the convexity of the 
problem~\eqref{mod-min}, a pair $(q,u) \in Q \times V_1$ is 
a minimum point if and only if there exists $p \in V_2$ such that 
\begin{equation}
\label{mod-opt-pre}
 Au = Cq,
\quad
 A^*p = I^*(Iu-\ud),
\quad
 \alpha q = - C^*p.
\end{equation}
We may eliminate $q$ by inserting the last equation into the first one and 
multiplying the second equation by $\beta>0$.  We thus obtain the
following \emph{reduced optimality system} for the pair $(u,p)\in
V_1\times V_2$:
\begin{equation}
\label{mod-ropt-pre}
 \left(\begin{matrix}
  - \beta I^* I & \beta A^* \\
  A & \tfrac{1}{\alpha} CC^* 
 \end{matrix}\right)
 \left(\begin{matrix}
 u \\
 p
 \end{matrix}\right)
 =
 \left(\begin{matrix}
  - \beta I^*\ud \\
  0
 \end{matrix}\right).
\end{equation}
Notice that the second row of equations, $ Au + \tfrac{1}{\alpha} CC^*
p=0$, suggests scaling the adjoint state $p$ by the factor
$\frac{1}{\alpha}$, while the first row, $-\beta I^* I u + \beta A^* p
= - \beta I^*\ud$, suggests no scaling at all. 
As a compromise, we propose to use $z=\tfrac{1}{\sqrt{\alpha}}  p$ and $\beta=\tfrac{1}{\sqrt{\alpha}}$.

We thus transform the optimality system~\eqref{mod-opt-pre} into
\begin{equation}
\label{mod-opt}
 Au = Cq,
\quad
 A^*z = \tfrac{1}{\sqrt{\alpha}}  I^*(Iu-\ud),
\quad
 \sqrt{\alpha}q = - C^*z
\end{equation}
and the reduced optimality system~\eqref{mod-ropt-pre} into
\begin{equation}
\label{mod-ropt}
 \left(\begin{matrix}
  -\tfrac{1}{\sqrt{\alpha}}  I^* I &  A^* \\
  A & \tfrac{1}{\sqrt{\alpha}}  CC^* 
 \end{matrix}\right)
 \left(\begin{matrix}
 u \\
 z
 \end{matrix}\right)
 =
 \left(\begin{matrix}
  - \tfrac{1}{\sqrt{\alpha}}  I^* \ud \\
  0 
 \end{matrix}\right).
\end{equation}
This \emph{rescaled and reduced optimality system} deviates from the
usual KKT-for\-mu\-la\-t\-ion, but has an interesting structure.  As
the KKT-formulation, it is symmetric also for non-symmetric $A$. The
off-diagonal consists of two interrelated invertible operators, while
the diagonal entries are (semi-)definite, symmetric operators. Notice
that, upon inverting the rows, the roles of the diagonal and
off-diagonal can be exchanged. For the optimization problem
\eqref{sim-min}, the operator matrix is then diagonally dominant in
that $CC^*$ and $I^*I$ are compact operators. 

Let us give a weak formulation of the rescaled and reduced optimality
system. Its rows are equivalently written as 
\begin{subequations}
\label{mod-vred-rows}
\begin{align}
 &\forall \vphi_1\in V_1
 & a(\vphi_1,z) - \tfrac{1}{\sqrt{\alpha}} \sprod{Iu}{I\vphi_1}_W
 &=
-\tfrac{1}{\sqrt{\alpha}} \sprod{\ud}{I\vphi_1}_W,
\\
 &\forall \vphi_2\in V_2
 &a(u,\vphi_2) + \tfrac{1}{\sqrt{\alpha}} \sprod{C^*z}{C^*\vphi_2}_Q
 &=
 0,
\end{align}
\end{subequations}
and so we are led to introduce the Hilbert space
\[
 V \defas V_1 \times V_2
\quad\text{with}\quad
 \norm{v}
 \defas
 \left(
  \norm{v_1}_{1}^2 + \norm{v_2}_{2}^2
 \right)^{1/2},
 \quad
 v=(v_1,v_2) \in V,
\]
and the bilinear form $b\colon V\times V\to\R$ given by
\begin{subequations}
\label{mod-blin}
\begin{equation}
 b(v,\vphi)
 \defas
 a(v,\vphi) + \tfrac{1}{\sqrt{\alpha}} c(v,\vphi)
\end{equation}
with
\begin{align}
\label{vec-a}
 a(v,\vphi)
 &\defas
 a(v_1,\vphi_2) + a(\vphi_1,v_2),
\\
 c(v,\vphi)
 &\defas
 \sprod{C^*v_2}{C^*\vphi_2}_Q
  -\sprod{Iv_1}{I\vphi_1}_W
\end{align}
\end{subequations}
for $v=(v_1,v_2),\vphi=(\vphi_1,\vphi_2)\in V$.  Note that we use the
same letter $a$ for the bilinear form inducing the operator $A$ and
for the one in \eqref{vec-a}; this ``operator overloading'' should not
cause confusion when the domain is clear. If not, we shall distinguish
the two forms by writing $a_{|V_1\times V_2}$ or $a_{|V\times V}$. In
this notation, the variational formulation of the rescaled and reduced
optimality system~\eqref{mod-ropt} simply reads 
\begin{equation}
\label{mod-vred}
\begin{aligned}
 \text{find } x\in V \text{ such that}
\quad
 \forall \vphi \in V
\;\;
 b(x,\vphi) = -\tfrac{1}{\sqrt{\alpha}} \sprod{\ud}{I\vphi_1}_W.
\end{aligned}
\end{equation}
A pair $x = (u,z) \in X$ is a solution of~\eqref{mod-vred} if and only
if $(u,z)$ is a solution of \eqref{mod-vred-rows} if and only if the  
triple $(u,z,-\tfrac{1}{\sqrt{\alpha}}C^*z)\in V\times Q$ verifies the
rescaled optimality system~\eqref{mod-opt}.  Consequently, thanks to
the convexity of~\eqref{mod-min}, if $x=(u,z)\in V$ is a solution
of~\eqref{mod-vred}, then  
$(-\tfrac{1}{\sqrt{\alpha}}C^*z,u)\in Q \times V_1$ is a solution of
the original optimization problem~\eqref{mod-min}. 

Let us analyze the bilinear form $b=a + \frac{1}{\sqrt{\alpha}}c$.
We readily see that
\begin{equation}
\label{symmetry}
 \text{$a_{|V\times V}$, $c$, and so $b$ are symmetric},
\end{equation}
but $b$ is not coercive in general. Consider, for example, a set-up
where there exists $v=(v_1,v_2)\in V$ such that $\norm{Iv_1}_W >
\norm{C^*v_2}_Q$. Then $c$ is not coercive and so, even for $a$
coercive, also $b$ is not coercive for $\alpha>0$ sufficiently small. 

 In order to obtain further properties, let us first consider the
 contributions $a$ and $c$ separately. The bilinear form $c$ is
 closely related to the original minimization problem \eqref{mod-min}
 and its ``energy seminorm'' \eqref{enorm}. To see this, observe that,
 if $(u,z)\in V$ and $\sqrt{\alpha} q = -C^*z$, we have the
 correspondence 
\begin{equation*}
 \norm{Iu}_W^2 + \norm{C^*z}_Q^2 
 =
 \norm{Iu}_W^2 + \alpha\norm{q}_Q^2,
\end{equation*}
which motivates to introduce the seminorm 
\begin{equation}
\label{enorm'}
 \snorm{v}
 \defas
 \left( \norm{Iv_1}_W^2 + \norm{C^*v_2}_Q^2
 \right)^{1/2}
\end{equation}
on $V$. Thus, denoting by $Z$ the kernel of $\snorm{\cdot}$ and
realizing that the bilinear form $c$ is well-defined on the quotient
space $V/Z$, we see that 
\begin{equation}
\label{cont-inf-sup-c}
 \sup_{\snorm{v}=1,\,\snorm{\vphi}=1} |c(v,\vphi)|
 =
 1
 =
 \infimum_{\snorm{v}=1} \; \sup_{\snorm{\vphi}=1} c(v,\vphi),
\end{equation}
where the second identity relies on 
\begin{equation}
\label{inf-sup-for-c}
 c\big( (v_1,v_2), (-v_1,v_2) \big)
 =
 \norm{C^*v_2}_Q^2 + \norm{Iv_1}_W^2
 =
 \snorm{v}^2.
\end{equation}
Since 
\begin{equation}
\label{e-V-norms}
 \forall v \in V
\quad
 \snorm{v}
 \leq
 M \norm{v}
\end{equation}
with
\begin{equation*}
 M \defas \max\{M_I,M_C\},
\end{equation*}
the form $c$ is also continuous in $V$, with constant $M$.

The bilinear form $a_{|V \times V}$ inherits its continuity and nondegeneracy properties from $a_{|V_1 \times V_2}$. More precisely, we have
\begin{equation}
\label{cont-inf-sup-vec-a}
 \sup_{\norm{v}=1,\,\norm{\vphi}=1} |a(v,\vphi)|
 =
 M_a
\quad\text{and}\quad
 \infimum_{\norm{v}=1} \; \sup_{\norm{\vphi}=1} a(v,\vphi)
 =
 m_a
\end{equation}
with $M_a$ and $m_a$ from \eqref{ass-on-a}. While the first identity is straight-forward, the second one hinges on the inf-sup-duality (cf.\ Babu\v{s}ka~\cite{Babuska:71})
\begin{equation}
\label{inf-sup-duality}
 \infimum_{\norm{v_1}_1=1} \; \sup_{\norm{\vphi_2}_{2}=1} a(v_1,\vphi_2)
 =
 \infimum_{\norm{v_2}_2=1} \; \sup_{\norm{\vphi_1}_{1}=1} a(\vphi_1,v_2)
\end{equation}
for $a$ with domain $V_1\times V_2$.

Turning to the complete bilinear form $b$, we may sum up the continuity properties as follows: for all $v,\vphi \in V$, we have
\begin{equation}
\label{cont}
%
 |b(v,\vphi)|
 \leq
 M_a \norm{v} \norm{\vphi}
  + \frac{M}{\sqrt{\alpha}} \norm{v} \snorm{\vphi}
 \leq
  \norm{v} \norm{\vphi}_\alpha
\end{equation}
with
\begin{equation}
  \label{norm-alpha}
  \norm{\vphi}_\alpha
 \defas
 M_a \norm{\vphi} + \frac{M}{\sqrt{\alpha}} \snorm{\vphi}.
\end{equation}
Here we have equipped $V$ as trial space with $\norm{\cdot}$ and as test space with $\norm{\cdot}_\alpha$. The former is in accordance with our scopes in the error analyses below and the latter avoids in particular a dependence on $M/\sqrt{\alpha}$ of the continuity constant of $b$ and in the following bound for the right-hand side in \eqref{mod-vred}: for all $\vphi=(\vphi_1,\vphi_2)\in V$,
\begin{equation}
 \snorm{ \tfrac{1}{\sqrt{\alpha}} \sprod{\ud}{I\vphi_1}_W }
 \leq
 \frac{M_I}{\sqrt{\alpha}}\norm{u_d}_W \norm{\vphi_1}_1
 \leq
 \norm{u_d}_W \norm{\vphi}_\alpha.  
\end{equation}

The derivation of the nondegeneracy properties of the bilinear form
$b$ is more subtle. In order to establish the crucial inf-sup
condition \eqref{inf-sup-a}, let $\vphi=(\vphi_1,\vphi_2)\in V$ be
given.

In order to find a suitable $v=(v_1,v_2) \in V$, we combine the nondegeneracy properties of $a$ and $c$ in the ansatz
\begin{subequations}
\label{inf-sup-ansatz}	
\begin{equation}
 v
 =
 (w_1,w_2) + \gamma (-\vphi_1,\vphi_2),
\end{equation}
where $\gamma\geq0$ and $w=(w_1,w_2)\in V$ is chosen with the help of \eqref{cont-inf-sup-vec-a} such that $\norm{w} = \norm{\vphi}$ and $a(w,\vphi) \geq m_a\norm{\vphi}^2$. We then have
\begin{equation}
\label{inf-sup-ansatz;norm-test-fct}
 \norm{v}
 \leq
 \norm{w} + \gamma \norm{\vphi}
 \leq
 ( 1+ \gamma) \norm{\vphi}
\end{equation}
and 
\begin{equation}
\begin{aligned}
 b(v,\vphi)
 &\geq
 m_a \norm{\vphi}^2
 +
 \frac{\gamma}{\sqrt{\alpha}} \snorm{\vphi}^2
 -
 \frac{M}{\sqrt{\alpha}} \snorm{\vphi} \norm{\vphi}
\\
 &\geq
 m_a \left(
  \norm{\vphi} + \frac{M}{M_a\sqrt{\alpha}} \snorm{\vphi}
 \right) \norm{\vphi}
 +
 \frac{\gamma}{\sqrt{\alpha}} \snorm{\vphi}^2
 -
 \frac{2M}{\sqrt{\alpha}} \snorm{\vphi}\norm{\vphi}.
\end{aligned}
\end{equation}
\end{subequations}
thanks the continuity \eqref{cont-inf-sup-c} of $c$ and $m_a\leq M_a$.
Using the inequality $2st \leq \epsilon s^2 + t^2/\epsilon$ with $\epsilon=\frac{L}{1+2L} m_a > 0$ and
\begin{subequations}
\label{def-kappa}
\begin{equation}
\label{L}
 L := M/\sqrt{\alpha},
\end{equation}
we may bound the critical term by
\begin{align*}
 \frac{2M}{\sqrt{\alpha}} \snorm{\vphi}^2
 \leq
 \frac{L}{1+2L} m_a \norm{\vphi}^2
 +
 \frac{1+2L}{L} \frac{M^2}{m_a\alpha}\snorm{\vphi}^2.
\end{align*}
Thus, if we define
\begin{equation}
\label{gamma}
 \gamma
 :=
  \frac{M}{m_a} \left( 1 + \frac{2M}{\sqrt{\alpha}} \right)  
\end{equation}
by the coefficient of $\snorm{\vphi}^2$ divided by $\sqrt{\alpha}$, set
\begin{equation}
\label{kappa}
 \kappa
 :=
 \frac{1+2L}{1+L} (1 + \gamma)
 =
 \frac{1+2L}{1+L} \left(
  1 + \frac{M}{m_a} \left( 1 + \frac{2M}{\sqrt{\alpha}} \right)
 \right),
\end{equation}
\end{subequations}
and recall \eqref{inf-sup-ansatz;norm-test-fct}, we arrive at
\begin{equation}
\label{inf-sup-const2}
\begin{aligned}
  b(v,\vphi)
  \geq
  \frac{1+L}{1+2L} \frac{m_a}{M_a} \norm{\vphi}_\alpha \norm{\vphi} 
  \geq
  \frac{1}{\kappa} \frac{m_a}{M_a} \norm{v} \norm{\vphi}_\alpha,
\end{aligned}
\end{equation}
where the norms on the right-hand side coincide with those in the continuity bound \eqref{cont}. We therefore have the following basic result.

\begin{thm}[Bilinear form of reduced optimality system]
\label{T:cont-inf-sup}
If we equip $V$ as trial space with $\norm{\cdot}$ and as test space with $\norm{\cdot}_\alpha$, then the inf-sup constant $m_b$ and the continuity constant $M_b$ of the bilinear form \eqref{mod-blin} satisfy
\begin{equation*}
 0
 <
 \frac{1}{\kappa} 
 	\frac{m_a}{M_a}
 \le
 m_b
 \le
 M_b
 \le 1,
\end{equation*}
where $\kappa$ is defined by the relations \eqref{def-kappa}.
\end{thm}

The inequalities of Theorem~\ref{T:cont-inf-sup} yield for the condition number of the bilinear form $b$
(i.e., the ratio of its continuity constant to its inf-sup constant)
\begin{equation*}
 \frac{M_b}{m_b}
 \leq
 \kappa \frac{M_a}{m_a}.
\end{equation*}
The second factor, the condition number of the bilinear form $a$
associated with the constraint, is expected to be a kind of lower
bound. In this vein, we may view the first factor $\kappa$ as a bound
for the possible amplification of the constraint conditioning,
resulting from the interplay of constraint and the objective in the
constrained optimization problem \eqref{mod-min}. Inspecting
\eqref{def-kappa}, we see that $\kappa$ is a function of the
parameters $\alpha$, $M$, $m_a$, and $M_a$. The next three remarks
discuss asymptotic behaviors of $\kappa$ that will play major roles in
what follows or are of independent interest. 

\begin{rem}[Amplification for pure constraint case]
\label{R:pure-constraint-cond}
Consider the special case $C=0$ and $I=0$. Then the rescaled and
reduced optimality system \eqref{mod-ropt} is a well-posed `double'
boundary value problem.  Its condition number with respect to
$(V,\norm{\cdot})\times(V,\norm{\cdot})$ is $M_a/m_a$; cf.\
\eqref{cont-inf-sup-vec-a}.
As $C=0$ and $I=0$ imply $M=0$, $L=0$,
and so $\gamma=0 $ and $\kappa=1$, this is reproduced by
Theorem~\ref{T:cont-inf-sup}.

It is worth mentioning that this limiting case of ``pure constraint'' is attained in a continuous manner:
\begin{equation*}
 \kappa - 1 = \big( 1 + o(1) \big) \frac{M}{\sqrt{\alpha}}
\quad\text{as}\quad
 M\to0,
\end{equation*}
where $L = M/\sqrt{\alpha}$ is essentially the operator norm of the perturbation.
\end{rem}

\begin{rem}[Amplification for degenerating constraint]
\label{R:deg-constraint}
While the continuity constant $M_a$ of the bilinear form $a$ does not
enter $\kappa$, its inf-sup constant $m_a$ does, in a critical manner. More precisely, we have 
\begin{equation*}
 \kappa
 =
 \left( 
  \frac{1+2L}{1+L} \left(
   1 + \frac{2M}{\sqrt{\alpha}}
  \right) M + o(1)
 \right)
 \frac{1}{m_a}
\quad\text{as}\quad
 m_a \to 0.
\end{equation*} 	
Notice that the fraction involving $L$ has only values in the
interval $[1,2]$.
\end{rem}

\begin{rem}[Amplification for vanishing regularization]
\label{R:vanish-regularization}
Consider the limit  $\alpha\to0$ of the Tikhonov regularization
parameter (while $I$ and $C$ are fixed).
Then $L \to \infty$ so that 
\begin{equation}
\label{kappa;vanish-regularization}
 \kappa
 =
 \left( \frac{4M^2}{m_a} + o(1) \right) \frac{1}{\sqrt{\alpha}}
\quad\text{as}\quad
 \alpha \to 0.
\end{equation}

Let us see with a simple example that the inf-sup constant $m_b$ in
Theorem~\ref{T:cont-inf-sup} can blow up with this rate and so the
lower bound therein cannot be improved for small $\alpha$ without
further assumptions on the structure of $b$. 

Consider $V_1=V_2=\R^2$, where $\norm{\cdot}_1$ and $\norm{\cdot}_2$
are the Euclidean norm in $\R^2$, 
\[
 \mathbf{A}
 =
 \left( \begin{array}{cc}
	1 & 0 \\
	0 & 1
 \end{array} \right),
\ \ 
 \mathbf{I}
 =
 \left( \begin{array}{cc}
	1 & 0 \\
	0 & 0
 \end{array} \right)
\text{ and } \
 \mathbf{C}
 =
 \left( \begin{array}{cc}
	0 & 0 \\
	0 & 1
 \end{array} \right).
\]
and $\alpha>0$. The symmetric bilinear form $b$ of the optimality
system is then given by the matrix 
\[
 \mathbf{B}
 = 
 \left( \begin{array}{cccc}
	-\tfrac1{\sqrt{\alpha}} & 0 & 1 &0 \\
	0 & 0 & 0 & 1\\
	1 & 0 & 0 & 0\\
	0 & 1 & 0 & \tfrac1{\sqrt{\alpha}} \\
 \end{array} \right).
\]
For $\vphi_0=(\sqrt\alpha, 0, 1, 0)\in V=\R^4$, we have
$\norm{\vphi_0}_\alpha = \sqrt{1+\alpha} + 1$ and 
\[
 \sup_{v\in V}  \frac{Bv\cdot\vphi_0}{\norm{v}}
 =
 \sup_{v\in V}  \frac{v\cdot(0,0,\sqrt{\alpha},0)}{\norm{v}}
 =
 \sqrt{\alpha}
\]
so that
\begin{equation}
\label{sharpness-of-alpha-blow-up}
 \infimum_{\vphi\in V} \sup_{v\in V} \frac{Bv\cdot\vphi}{\norm{v} \norm{\vphi}_\alpha}
 \leq
 \sqrt{\frac{\alpha}{2}}.
\end{equation}
Hence, the asymptotic behavior of $\alpha$ in \eqref{kappa;vanish-regularization} is attained.
\end{rem}

The chosen norms for $V$ as trial and test space are not always the
most convenient ones. This follows from the following remark
considering a special case. 
\begin{rem}[Coercive constraints with $C^*=I$]
\label{R:coercivityC^*=I;cont-inf-sup}
Suppose that $V_1=V_2$ and $Q=W$ with coinciding scalar products and
norms and that the bilinear form $a_{|V_1\times V_1}$ is coercive with
constant $\tilde{m}_a$ and $C^*=I$. It is worth noting that, as
$a_{|V_1 \times V_1}$ is not necessarily symmetric, the best
coercivity constant $\tilde{m}_a$ may be much smaller than the inf-sup
constant $m_a$.  Given $\vphi \in V$, we proceed as in
\eqref{inf-sup-ansatz} taking $w = \vphi$, $\gamma = 0$, and
obtain
\begin{subequations}
\label{alternative-norms}
\begin{equation}
 b(v,\vphi)
 \geq
 \tilde{m}_a \left( \norm{\vphi}^2 + \frac{1}{\sqrt{\alpha}}\snorm{\vphi}^2 \right)
\end{equation}
because of $c(\vphi,\vphi)=0$. This fits well to the following variant of the continuity bound \eqref{cont}:
\begin{equation}
 |b(v,\vphi)|
 \leq \max\{ M_a, 1 \}
  \left( \norm{v}^2 +\frac{1}{\sqrt{\alpha}} \snorm{v}^2 \right)^{1/2}
  \left( \norm{\vphi}^2 +\frac{1}{\sqrt{\alpha}} \snorm{\vphi}^2 \right)^{1/2}.
\end{equation}
\end{subequations}
Hence, in this case, the condition number of $b$ with respect to the
norms in \eqref{alternative-norms} is independent of the Tikhonov
regularization parameter $\alpha$. Nevertheless, if $C^*\neq I$, also
this choice of norms cannot offer in general an asymptotic behavior
better than $1/\sqrt{\alpha}$ as $\alpha\to0$. In fact, re-computing
the example in Remark \ref{R:vanish-regularization} with the norms in
\eqref{alternative-norms} does not change the behavior of its inf-sup
constant. 
\end{rem}

Let us conclude this section with the following side product of our
discussion of the bilinear form $b$. 
\begin{cor}[Existence and uniqueness]
\label{C:uniqueness} 
The rescaled and reduced optimality system~\eqref{mod-vred} and
thus~\eqref{mod-opt-pre} has a unique solution. 
\end{cor}
\begin{proof}
Inequality \eqref{inf-sup-const2} ensures \eqref{inf-sup-a} for the
bilinear form $b$ and, thanks to the algebraic symmetry of $b$, also
\eqref{non-deg-a}. 
\end{proof}

\section{Analysis for variational discretization}\label{sec:3}
In this section, we analyze the error of the variational
discretization of the optimization problem \eqref{mod-min} according
to Hinze~\cite{Hinze:2005}. Our key tool is the rescaled and reduced
optimality system \eqref{mod-ropt}, whose Galerkin solution coincides
with the approximate solution of the variational discretization. 
\subsection{Variational discretization and reduced optimality system}
\label{sec:vardisc}
We start by discretizing the PDE constraint \eqref{bvp} of the
optimization problem \eqref{mod-min}. Recalling its variational
formulation 
\begin{equation*}
 \text{find }u\in V_1
\quad\text{such that}\quad
 \forall \vphi_2 \in V_2 \;\;
 a(u,\vphi_2) = \dual{Cq}{\vphi_2},
\end{equation*}
we choose some conforming finite-di\-men\-sion\-al spaces $V_{h,i}
\subset V_i$, $i=1,2$, such that the restriction of the bilinear form
$a$  on $V_{h,1}\times V_{h,2}$ is nondegenerate. The corresponding
Petrov-Galerkin method then reads 
\begin{equation*}
 \text{find }u_h\in V_{h,1}
\quad\text{such that}\quad
 \forall \vphi_{h,2}\in V_{h,2} \;\;
 a(u_h, \vphi_{h,2}) = \dual{Cq}{\vphi_{h,2}}.
\end{equation*}
Using this for the constraint in~\eqref{mod-min}, we arrive at the (semi-)discrete optimization problem
\begin{equation}
\label{mod-min-h}
\begin{aligned}
 \min_{(\Tilde{q},u_h) \in Q\times V_{h,1}}
  &\;\frac{1}{2}\norm{Iu_h-\ud}_W^2 + \frac{\alpha}{2}\norm{\Tilde{q}}^2_Q
\\
 &\text{subject to}
\qquad
 \forall \vphi_{h,2} \in V_{h,2}
\quad
 a(u_h, \vphi_{h,2})
 =
 \sprod{\Tilde{q}}{C^*\vphi_{h,2}}_Q,
\end{aligned}
\end{equation}
where we, in addition, assume that $I$ can be exactly evaluated for any function 
from $V_{h,1}$.  As in the continuous case,  $(\Tilde{q},u_h) \in Q \times V_{h,1}$ is 
the unique solution of~\eqref{mod-min-h} if and only if there exists $z_h \in V_{h,2}$ such that
\begin{equation}
\label{mod-opt-h}
\begin{aligned}
 &\forall \vphi_{h,2} \in V_{h,2}
 &a(u_h,\vphi_{h,2})
 &=
 \sprod{\Tilde{q}}{C^*\vphi_{h,2}}_Q,
\\
 &\forall \vphi_{h,1} \in V_{h,1}
 &a(\vphi_{h,1},z_h)
 &=
 \tfrac{1}{\sqrt{\alpha}}  \sprod{Iu_h-\ud}{I \vphi_{h,1}}_W, 
\\
 &&\sqrt{\alpha} \Tilde{q} &= - C^*z_h. 
\end{aligned}
\end{equation}
Also here, we may eliminate the approximate control $\Tilde{q}$ by inserting the 
third equation into the first one.  Setting $V_h \defas V_{h,1} \times V_{h,2}$, the 
variational formulation of the ensuing discrete rescaled and reduced optimality system is
\begin{equation}
\label{mod-ropt-h}
\begin{aligned}
 \text{find } x_h&=(u_h,z_h) \in V_h \text { such that }
\\
 &\forall \vphi_h = (\vphi_{h,1},\vphi_{h,2}) \in V_h
\quad
 b(x_h,\vphi_h)
 =
 - \tfrac{1}{\sqrt{\alpha}}  \sprod{\ud}{I\vphi_{h,1}}_W.
\end{aligned}
 \end{equation}
Its solution $x_h$ is the Galerkin approximation in $V_h$ to the solution $x$ 
of the variational formulation~\eqref{mod-vred} of the rescaled and reduced optimality 
system.  Applying Corollary~\ref{C:uniqueness} to the 
discrete spaces therefore yields the following approach to uniqueness and 
existence of the variational discretization of~\eqref{mod-vred}. 
\begin{lem}[Discrete well-posedness]
\label{L:well-posedness-h}
The discrete reduced optimality system~\eqref{mod-ropt-h} has a unique 
variational solution $x_h=(u_h,z_h)\in V_h$.  Consequently, the pair 
$(\Tilde{q},u_h)$ with $\Tilde{q}=-\tfrac{1}{\sqrt{\alpha}} C^*z_h$ is the unique solution of 
the semidiscrete optimization problem~\eqref{mod-min-h}.
\end{lem}
Remarkably, the approximate solutions $(\Tilde{q},u_h,z_h)$ of the variational discretization~\eqref{mod-opt-h} are computable 
whenever $C^*_{|V_{h,2}}$ and $I_{|V_{h,1}}$ can be evaluated exactly.

\subsection{Non-asymptotic quasi-best approximation}
\label{S:qb-approx}
We shall assess the quality of the Galerkin approximation $x_h=(u_h,z_h) \in
V_h$ from \eqref{mod-ropt-h}, assuming that we are  interested
particularly in the $\norm{\cdot}_1$-error of the approximate state
$u_h$. For this purpose, we compare it with a suitable best error in
$V_h$. 

Let us first recall some basic results in Petrov-Galerkin
approximation, which we already formulate for the discretization of
the constraint. Let $R_{h,1}v_1 \in V_{h,1}$ be the generalized Ritz
projection of $v_1\in V_1$ given by $a(R_{h,1}v_1,\vphi_{h,2}) =
a(v_1,\vphi_{h,2})$ for all $\vphi_{h,2} \in V_{h,2}$. Since
$a_{|V_1\times V_2}$ satisfies \eqref{ass-on-a} and is nondegenerate
on $V_{h,1} \times V_{h,2}$, there exists a constant $\mu_h\geq1$ such
that 
\begin{equation*}
 \norm{v_1 - R_{h,1}v_1}_1
 \leq
 \mu_h \inf_{v_{h,1} \in V_{h,1}} \norm{ v_1 - v_{h,1} }_1;
\end{equation*}
see, e.g., Babu\v{s}ka \cite{Babuska:71}. We refer to the smallest possible choice of $\mu_h$ as the \emph{quasi-best-approximation constant of the constraint discretization}. Xu and Zikatanov \cite{XuZikatanov:03} show the identities
\begin{equation}
 \mu_h = \norm{I-R_{h,1}}_{L(V_{h,1})} =  \norm{R_{h,1}}_{L(V_{h,1})}
\end{equation}
and Tantardini and Veeser \cite[Theorem~2.1]{Tantardini.Veeser:16} give the formula
\begin{equation}
\label{qba-const}
 \mu_h = \sup_{\vphi_{h,2} \in V_{h,2}}
  \frac{\sup_{\norm{v_1}_{1}=1} a(v_1,\vphi_{h,2})}{\sup_{\norm{\smash{v_{h,1}}}_{1}=1} a(v_{h,1}, \vphi_{h,2})},
\end{equation}
where $v_1$ varies in $V_1$ and $v_{h,1}$ varies in $V_{h,1}$ and, for the sake of notational simplicity, a tedious $\vphi_{h,2}\neq0$ is avoided.

A perhaps striking feature of these formulas is that they are not affected by the choices of the norms in the test spaces $V_{h,2}$ and $V_2$. This comes in quite useful in our context, as the adjoint state is an auxiliary variable and, in the original approximation problem \eqref{mod-min}, the norm $\norm{\cdot}_2$ is free as long as \eqref{ass-on-a} continues to hold with $\norm{\cdot}_1$. Exploiting this freedom, we propose to (possibly) redefine the norm on the space $V_2$ by
\begin{equation}
\label{adjoint-norm}
 \norm{v_2}_2
 \defas
 \sup_{\vphi_1 \in V_1, \norm{\vphi_1}_1=1} a(\vphi_1, v_2)
\end{equation}
and so, in particular, to measure the error of the approximate adjoint state $z_h$ in this norm. This redefinition affects the constants that we associated with the constrained optimization problem \eqref{mod-min}. The new continuity and inf-sup constants of the bilinear forms $a_{|V_1\times V_2}$ are
\begin{equation}
\label{new-infsup-cont-a}
M_a = 1 = m_a.
\end{equation}
The constant $M_I$ is not affected, while we have
\begin{equation}
\label{new-vs-old}
 M_{a,\text{old}}^{-1}
 \leq
 \frac{M_C}{M_{C,\text{old}}}
 \leq
 m_{a,\text{old}}^{-1}.
\end{equation}
where we indicate quantities before the redefinition by an additional index ``old''. As in addition
\begin{equation*}
  m_{a,\text{old}} \norm{\cdot}_{2,\text{old}}
 \leq
 \norm{\cdot}_2
 \leq
 M_{a,\text{old}} \norm{\cdot}_{2,\text{old}},
\end{equation*}
the results below hold also with the original norm in $V_2$, but the constants have to be revisited.

The convenience of the choice \eqref{adjoint-norm} lies in the following consequences of \eqref{new-infsup-cont-a}. The numerator in \eqref{qba-const} is $\norm{\vphi_{h,2}}_2$, which, together with the inf-sup-duality, cf.\ \eqref{inf-sup-duality}, yields 
\begin{equation}
\label{inf-sup-ah}
 \infimum_{v_{h,1} \in V_{h,1}} \ \sup_{\vphi_{h,2} \in V_{h,2}}
\frac{a(v_{h,1},\vphi_{h,2})}{\norm{v_{h,1}}_1 \norm{\vphi_{h,2}}_2}
 =
 \frac{1}{\mu_h}
 =
  \infimum_{v_{h,2} \in V_{h,2}} \ \sup_{\vphi_{h,1} \in V_{h,1}}
 \frac{a(\vphi_{h,1},v_{h,2})}{\norm{v_{h,2}}_2 \norm{\vphi_{h,1}}_1}
\end{equation}
for the inf-sup constant of $a_{|V_{h,1}\times V_{h,2}}$. Accordingly, the generalized Ritz projection $R_{h,2}v_2 \in V_{h,2}$ of $v_2\in V_2$ given by $a(\vphi_{h,1}, R_{h,2}v_2) = a(\vphi_{h,1},v_2)$ for all $\vphi_{h,1} \in V_{h,1}$ verifies
\begin{equation*}
 \norm{v_2 - R_{h,2}v_2}_2
 \leq
 \mu_h \inf_{v_{h,2} \in V_{h,2}} \norm{ v_2 - v_{h,2} }_2.
\end{equation*}
Setting $R_h=(R_{h,1}, R_{h,2})$, we also have
\begin{equation}
\label{qba-Rh}
 \norm{v - R_h v}
 \leq
 \mu_h \inf_{v_h \in V_h} \norm{ v - v_h }.
\end{equation}

After these preparations, we are ready to derive a first result about quasi-best approximation of the variational discretization \eqref{mod-min-h}.

\begin{thm}[Non-asymptotic quasi-best approximation]
\label{T:qba}
Let $x=(u,z)$ be any solution of the optimality system \eqref{mod-vred} and choose \eqref{adjoint-norm} as norm in $V_2$. The combined error in the corresponding approximate state $u_h$ and its adjoint $z_h$ of the variational discretization is quasi-best in $V_h$ with
\begin{align*}
 \norm{x-x_h}
 \le
 \kappa_h \mu_h \inf_{v_h\in V_h}\|x-v_h\|.
\end{align*}
Here
\begin{equation*}
 \kappa_h
 =
 \frac{1+2L}{1+L} \left(
  1 + M \left( 1 + \frac{2M}{\sqrt{\alpha}} \right) \mu_h
 \right)
\quad\text{with}\quad
 L
 =
 \frac{M}{\sqrt{\alpha}},
\end{equation*}
and $\mu_h$ is the quasi-best-approximation constant of the constraint discretization.
\end{thm}

\begin{proof}
Thanks to Theorem \ref{T:cont-inf-sup} and Lemma \ref{L:well-posedness-h}, we can use the counterpart of \eqref{qba-const} for the characterization \eqref{mod-ropt-h} of the variational discretization. Let $\vphi_h \in V_h$. The continuity bound \eqref{cont} and \eqref{new-infsup-cont-a} give for the numerator
\begin{equation*}
 \sup_{\norm{v}=1} b(v,\vphi) \leq \norm{\vphi_h}_\alpha.
\end{equation*}
For the denominator, we use \eqref{inf-sup-ansatz}, where $V$ is replaced by $V_h$ and, therefore, with $1/\mu_h$ in place of $m_a$ in view of \eqref{inf-sup-ah}. We thus obtain
\begin{equation}
\label{infsup-b-Vh}
 \sup_{v_h \in V_h, \norm{v_h}=1} b(v_h,\vphi_h)
 \geq
 \frac{1}{\kappa_h\mu_h}\norm{\vphi_h}_\alpha
\end{equation}
and the proof is finished.
\end{proof}

In the special situation of Remark \ref{R:coercivityC^*=I;cont-inf-sup}, we can obtain the following quasi-best approximation result.
\begin{rem}[Quasi-best approximation for coercive constraints and $C^*=I$]
\label{R:coercivityC^*=I;qba}
Suppose that $V_1=V_2$ and $Q=W$ with coinciding scalar products and norms and that the bilinear form $a$ is $V_1$-coercive with constant $\tilde{m}_a$ and $C^*=I$. Exploiting the coercivity and continuity properties of Remark \ref{R:coercivityC^*=I;cont-inf-sup}, we derive for the error of the variational discretization \eqref{mod-vred}
\begin{align*}
 \norm{x-x_h}^2 + \frac{1}{\sqrt{\alpha}} \snorm{x-x_h} ^2
 \le
 \frac{\max\{M_a^2,1\}}{\tilde{m}_a^2}
  \inf_{v_h\in V_h} \left(
   \norm{x-v_h}^2 + \frac{1}{\sqrt{\alpha}} \snorm{x-v_h}^2
  \right).
\end{align*}
\end{rem}

The quasi-best approximation constant in the preceding
Remark~\ref{R:coercivityC^*=I;qba} does not blow up for vanishing
regularization. Nonetheless, when measuring the error merely with
$\norm{\cdot}$, it does not exclude an $\alpha^{-1/4}$-blow up of the
quasi-best approximation constant even in the special case $C^*=I$
considered in Remark~\ref{R:vanish-regularization} and, in the light
of the example therein, it does not exclude an $\alpha^{-3/4}$-blow up
for general operators $I$ and $C$. As we shall see, the
$\alpha$-dependence in Theorem \ref{T:qba} is less severe. 

\begin{rem}[Vanishing regularization and quasi-best approximation]
\label{R:vanish-regu-qba}
As in Remark \ref{R:vanish-regularization}, we consider the limit $\alpha\to0$ for the Tikhonov regularization parameter. Similarly to there, we have
\begin{equation}
\label{kappa_h;alpha}
 \kappa_h
 =
 \left( \frac{4M^2}{\mu_h} + o(1) \right) \frac{1}{\sqrt{\alpha}}
 \quad\text{as}\quad
 \alpha \to 0.
\end{equation}
This blow up arises from the lower bound of the inf-sup constant in
Theorem \ref{T:cont-inf-sup}, which cannot be improved because of
\eqref{sharpness-of-alpha-blow-up}.
Note however, that the equivalence of the norms $\norm{\cdot}_\alpha$
and $\sup_{\norm{v}=1} b(v,\cdot)$ is not uniform in $\alpha$. In the
light of \eqref{qba-const}, it is therefore conceivable that
\eqref{kappa_h;alpha} could be improved by using the latter as test
space norm. However, the determination of the discrete inf-sup
constant with respect to this abstract norm appears to be much more
involved than the approach \eqref{inf-sup-ansatz}, which directly
carries over to discrete spaces.  
\end{rem}

In any case, we shall show below that, under refinement, the $\alpha$-dependence disappears for many instances of the optimality system \eqref{mod-opt}.

\subsection{Asymptotic quasi-best approximation}
\label{S:asym-qba}
%
In this section, we complement Theorem~\ref{T:qba}. To be more precise,
let $\nu_h$ be the \emph{quasi-best-ap\-prox\-i\-ma\-tion constant of the variational discretization} therein and consider a sequence $(V_h)_h$ of discrete spaces leading to a uniform stable constraint discretization in that
\begin{align}
\label{unif-qba}
\exists \bar{\mu} \geq 1 \quad \forall h>0 \quad \mu_h \leq \bar{\mu},
\end{align}
which is equivalent to discrete inf-sup stability in view of \eqref{inf-sup-ah}.
Theorem \ref{T:qba} then ensures the existence of a constant $\bar{\nu}$ such that
\begin{equation}
\label{nu_h<nu}
 \forall h>0 \quad \nu_h \leq \bar{\nu}.
\end{equation}
This upper bound may be pessimistic. To motivate this assessment, represent the bilinear form $b$ by the operator matrix
\begin{equation*}
\left(\begin{matrix}
A & \tfrac{1}{\sqrt{\alpha}}  CC^* \\
-\tfrac{1}{\sqrt{\alpha}}  I^* I &  A^* 
\end{matrix}\right),
\end{equation*}
which is the one in \eqref{mod-ropt} with inverted rows. If $C$ and $I$ are compact, this matrix is diagonally dominant in an operator sense and can be viewed as a compact perturbation of the diagonal matrix with the entries $A$ and $A^*$. Therefore, in order to improve on \eqref{nu_h<nu}, we mimic somewhat the argument in Schatz \cite{Schatz:1974}, introducing some new twist.

Let us first observe that, in accordance with Remark \ref{R:pure-constraint-cond}, Theorem~\ref{T:qba} yields $\nu_h\leq\mu_h$ whenever $M_I=0=M_C$. More precisely and generally, we have the following relationship between the two quasi-best-approximation constants.
\begin{lem}[Quasi-best-approximation constants]
\label{L:qba-consts}
The quasi-best-approximation constants $\nu_h$ and $\mu_h$ are related by
\begin{equation*}
 |\nu_h-\mu_h|
 \leq
 \kappa_h \mu_h \sup_{\norm{v}=1} \snorm{v-R_hv},
\end{equation*}
where $\kappa_h$ is as in Theorem \ref{T:qba} and $R_h$ is the generalized Ritz projection in \eqref{qba-Rh}.
\end{lem}

\begin{proof}
As in the proof of Theorem \ref{T:qba}, we will make use of \eqref{qba-const} with $a$ replaced by $b$.  
Given $v\in V$ and $\vphi_h \in V_h$, we can write
\begin{equation*}
 b(v,\vphi_h)
 =
 b(R_hv,\vphi_h)
 +
 \frac{1}{\sqrt{\alpha}} c(v-R_hv,\vphi_h)
\end{equation*}
because of $a(v-R_hv,\vphi_h)=0$. Hence,
\begin{equation*}
 \snorm{ \sup_{\norm{v}=1} b(v,\vphi_h) - \sup_{\norm{v}=1} b(R_hv,\vphi_h) }
 \leq
 \frac{1}{\sqrt{\alpha}} \sup_{\norm{v}=1} |c(v-R_hv,\vphi_h)|.
\end{equation*}
As
\begin{equation*}
 \frac{\sup_{\norm{v}=1} b(R_hv,\vphi_h)}{\sup_{v_h\in V_h, \norm{v_h}=1} b(v_h,\vphi_h) }
 \leq
 \norm{R_h}_{L(V)} = \mu_h
\end{equation*}
with equality for some $\vphi_h\in V_h$, we obtain
\begin{equation*}
 |\nu_h - \mu_h|
 \leq
 \sup_{\vphi \in V_h} \frac{\frac{1}{\sqrt{\alpha}}\sup_{\norm{v}=1} |c(v-R_hv,\vphi_h)|}
  {\sup_{v_h\in V_h, \norm{v_h}=1} b(v_h,\vphi_h)}.
\end{equation*}
Thanks to \eqref{cont-inf-sup-c}, \eqref{norm-alpha}, and \eqref{infsup-b-Vh} this proves the claimed inequality.
\end{proof}

In order to deploy Lemma \ref{L:qba-consts}, we need additional assumptions for our optimization problem and its discretization. We shall consider two settings: a ``qualitative'' and a ``quantitative'' one. The former assumes in addition
\begin{subequations}
\label{qualitative}
\begin{equation}
\label{compactness}
 \text{$I:V_1\to W$ and $C:Q\to V_2^*$ are compact}
\end{equation}
for the optimization problem and
\begin{align}
\label{approximability}
 \forall v \in V \quad \lim_{h\to 0} \inf_{v_h \in V_h} \norm{v-v_h} = 0,
\end{align}
\end{subequations}
for the constraint discretization. Notice that, owing to \eqref{new-vs-old}, the condition \eqref{compactness} is independent of our choice to equip $V_2$ with the norm \eqref{adjoint-norm}.
\begin{lem}[Qualitative asymptotic quasi-best approximation]
\label{L:aqba-qualitative}
Under the assumptions \eqref{unif-qba} and \eqref{qualitative}, the quasi-best-approximation constant $\nu_h$ satisfies
\begin{equation*}
 \nu_h = \mu_h \big( 1 + \bar{\kappa} \, o(1) \big)
\quad\text{as}\quad
 h\to 0,
\end{equation*}
where
\begin{equation*}
 \bar{\kappa}
  =
 \frac{1+2L}{1+L} \left(
 1 + M \left( 1 + \frac{2M}{\sqrt{\alpha}} \right) \bar{\mu}
 \right)
 \quad\text{with}\quad
 L
 =
 \frac{M}{\sqrt{\alpha}}.
\end{equation*}
\end{lem}

\begin{proof}
In the light of Lemma \ref{L:qba-consts} and \eqref{unif-qba}, it suffices to verify the uniform convergence
\begin{equation}
\label{unif-conv}
 \lim_{h\to 0} \sup_{\norm{v}=1} \snorm{v-R_hv} = 0.
\end{equation}
This follows from a standard argument; we provide details for the sake of completeness. Let $(h_k)_k$ be any sequence with $\lim_{k\to 0} h_k = 0$ and choose $v_k$ such that
\begin{equation*}
 \forall k\in\N
\quad
 \norm{v_k} = 1
 \text{ and } 
 \sup_{\norm{v}=1} \snorm{v-R_kv}
 \leq
 \snorm{v_k - R_k v_k} + \frac{1}{k},
\end{equation*}
where we write $k$ instead $h_k$ whenever the latter is an index. Exploiting \eqref{unif-qba} another time, we see that the sequence given by $d_k \defas v_k - R_k v_k$ is bounded in the Hilbert space $V$. Owing to \eqref{approximability}, its weak limit $d\in V$ satisfies
\begin{equation*}
 a(d,\vphi) = a(d-d_k,\vphi) + a(d_k, \vphi-\vphi_k)
\end{equation*}
for any $\vphi \in V$ and $\vphi_k \in V_k$. Choosing $\vphi_k$ by means of \eqref{approximability}, we derive $a(d,\vphi)=0$ by $k\to\infty$. Consequently, \eqref{cont-inf-sup-vec-a} yields $d=0$. Thanks to \eqref{compactness}, the operator $I:V_1\to W$ and the adjoint $C^*:V_2\to Q$ are compact. This turns the weak convergence $d_k \to 0$ in $V$ into the strong convergence $\snorm{d_k}\to0$ and the proof is finished.
\end{proof}

In order to quantify the convergence in Lemma \ref{L:aqba-qualitative}, we shall use a duality argument. This requires a second, more specific setting of additional assumptions involving the Sobolev spaces $H^s$, $s\geq0$, and their norms $\snorm{\cdot}_s$ over some domain. We use $\snorm{\cdot}_{s}$ instead of $\norm{\cdot}_{s}$ in order to avoid confusion with the norms $\norm{\cdot}_1$ and $\norm{\cdot}_2$ of $V_1$ and $V_2$. For $s<0$, we denote by $H^s$ the (topological) dual space of $H^{-s}$ and  $\snorm{\cdot}_s$ stands for the dual norm of $\snorm{\cdot}_{-s}$.

We suppose that spaces $V_1$ and $V_2$ relate to Sobolev spaces in the following way: There are $s_i\in\R$, $i=1,2$, and a constant $C_S\geq1$ such that
\begin{subequations}
\label{quantitative}
\begin{equation}
\label{Sobolev}
 \text{$V_i$ is a closed subspace of $H^{s_i}$ and $C_S^{-1}\snorm{\cdot}_{s_i} \leq \norm{\cdot}_i \leq C_S \snorm{\cdot}_{s_i}$ for $i=1,2$.}
\end{equation}
Furthermore, we suppose that there is $\delta>0$ such that the following three conditions hold. First, the operators $C$ and $I$ have the boundedness properties
\begin{equation}
\label{quant-compactness}
 C \in L(Q,H^{-s_2+\delta})
\quad\text{and}\quad
 I \in L(H^{s_1-\delta},W).
\end{equation}
Thus, the canonical embeddings $H^{-s_2+\delta} \to H^{-s_2}$ and $H^{s_1}\to H^{s_1-\delta}$ quantify the compactness assumption \eqref{compactness}. Second, the differential operator of the constraint and its adjoint offer the following regularity estimates: there is a constant $C_R>0$ such that, for all admissible $f$ and $g$,
\begin{equation}
\label{shift}
 \snorm{A^{-1}f}_{s_1+\delta}
 \leq
 C_R \snorm{f}_{-s_2+\delta}
\quad\text{and}\quad
 \snorm{A^{-*}g}_{s_2+\delta}
 \leq
 C_R \snorm{g}_{-s_1+\delta}.
\end{equation}
Third and last, the approximation spaces $V_h$ verify
\begin{equation}
\label{quant-approx}
 \inf_{v_h \in V_h} \norm{v-v_h}
 \leq
 C_{\mathcal{I}} h^\delta \left(
  \snorm{v_1}_{s_1+\delta}^2 + \snorm{v_2}_{s_2+\delta}^2
 \right)^{1/2}
\end{equation}
\end{subequations}
for some constant $C_{\mathcal{I}}>0$, which quantifies the approximation property \eqref{approximability}.

\begin{thm}[Quantitative asymptotic best approximation]
\label{T:aqba-quant}
Under the assumptions \eqref{unif-qba} and \eqref{quantitative}, the quasi-best-approximation constant $\nu_h$ satisfies
\begin{equation*}
 \nu_h = \mu_h \big( 1 + \bar{\kappa} O(h^\delta) \big)
\quad\text{as}\quad
 h \to 0,
\end{equation*}
where $\bar{\kappa}$ is as in Lemma \ref{L:aqba-qualitative}. For the $\alpha$-dependence of $\bar{\kappa}$, cf.\ Remark \ref{R:vanish-regu-qba}.
\end{thm}

\begin{proof}
Similarly as in the first step of the proof of Lemma~\ref{L:aqba-qualitative}, inserting \eqref{unif-qba} and
\begin{equation}
\label{quant-conv}
 \lim_{h\to 0} \sup_{\norm{v}=1} \snorm{v-R_hv} = O(h^\delta).
\end{equation}
into  Lemma~\ref{L:qba-consts} establishes the claim. To show \eqref{quant-conv}, let $v \in V$ with $\norm{v}=1$  and define $\vphi \in V$ as the solution of the following ``dual'' problem associated with the bilinear form $a_{|V\times V}$:
\begin{equation*}
 A\vphi_1 = C C^* d_2,
\quad
 A^* \vphi_2 = I^* I d_1,
\end{equation*}
where $d =(d_1,d_2) \defas v-R_hv$. We thus have
\begin{equation}
\label{duality-arg;1}
\begin{aligned}
 \snorm{v-R_hv}^2
 &=
 |d|^2
 =
 \langle I^* I d_1,d_1 \rangle_1 + \langle C C^* d_2, d_2 \rangle_2
 =
 a(d,\vphi)
 =
 a(v-R_hv,\vphi)
\\
 &=
 a(v-R_hv, \vphi-\vphi_h)
 \leq
 \norm{v-R_hv} \norm{\vphi-\vphi_h}, 
\end{aligned}
\end{equation}
where $\vphi_h \in V_h$ is arbitrary. For the first factor, \eqref{qba-Rh} and \eqref{unif-qba} imply
\begin{equation}
\label{duality-arg;2}
 \norm{v-R_hv} \leq \mu_h \leq \mu.
\end{equation}
For second factor, we employ \eqref{quant-approx} with suitable $\vphi_h\in V_h$ to obtain
\begin{align*}
 \norm{\vphi-\vphi_h}
 \leq
 C_{\mathcal{I}} h^\delta \left(
 \snorm{\vphi_1}_{s_1+\delta}^2 + \snorm{\vphi_2}_{s_2+\delta}^2
 \right)^{1/2}
\end{align*}
and it remains to show that the norms on the right-hand side are suitably bounded. Let consider the first one. Making use of the regularity estimate \eqref{shift} and the definition of $\vphi_1$, we deduce
\begin{align*}
 \snorm{\vphi_1}_{s_1+\delta}
 &\leq
 C_R \snorm{A\vphi_1}_{-s_2+\delta}
 =
 C_R \snorm{CC^*d_2}_{-s_2+\delta}
 \leq
 C_R \bar{M}_C \norm{C^*d_2}_Q
\\
 &\leq
 C_R \bar{M}_C\snorm{d}
 =
 C_R \bar{M}_C \snorm{v-R_hv},
\end{align*}
where $\bar{M}_C$ is the operator norm of $C$ from \eqref{quant-compactness}. A similar argument yields
\begin{equation*}
 \snorm{\vphi_2}_{s_2+\delta}
 \leq
 C_R \bar{M}_I \snorm{v-R_hv},
\end{equation*}
where $\bar{M}_I$ is the operator norm of $I$ in \eqref{quant-compactness}. We insert the previous estimates in the first one and conclude
\begin{equation*}
 \snorm{v-R_hv} \leq \mu C_{\mathcal{I}} C_R \bar{M} h^\delta
\end{equation*}
with $\bar{M} \defas \max\{\bar{M}_I, \bar{M}_C \}$, i.e., \eqref{quant-conv}.
\end{proof}
 
Let us exemplify Theorem \ref{T:aqba-quant} by two applications. The first one considers the optimization problem \eqref{sim-min} of the introduction, while the second one is more involved in that the constraint does not allow for a coercive set-up.

\begin{ex}[Simple model optimization]
\label{ex:sim-opt}
Discretize the optimization problem \eqref{sim-min} of the introduction with linear finite elements on quasi-uniform meshes with meshsize $h$. We have $V_1 = H^1_0 = V_2$ and, if we choose $\norm{\cdot}_1 = \snorm{\nabla\cdot}_0$, we already have $m_a=1=M_a$ and \eqref{adjoint-norm} does not change the norm in $V_2$. Further, $M_I=C_F=M_C$, where $C_F$ is the constant in the Poincar\'e-Friedrichs inequality.  Moreover, we have $s_1=1=s_2$ and, assuming that the underlying domain is convex, $\delta=1$. Taking Sobolev seminorm instead of norms in \eqref{Sobolev}, we then have $C_S=1$ for the relevant cases and $C_R=1$ thanks to elliptic regularity as well as $\bar{M}_I=1=\bar{M}_C$. Standard approximation theory shows \eqref{quant-approx} with $C_\mathcal{I}$ depending on the shape regularity of the underlying meshes. Since $\mu_h=1$, we conclude
\begin{equation*}
 |\nu_h - 1|
 \leq 
 2 \left( 1 + C_F \left( 1 + \frac{2 C_F}{\sqrt{\alpha}} \right)\right) h
\quad\text{as}\quad
 h\to0
\end{equation*}
for the quasi-best-approximation constant of the variational discretization in this case.
\end{ex}
\begin{ex}[Point source control]
	\label{ex:deltas}
	We consider the following modification of the optimization problem \eqref{sim-min}, where the distributed control is replaced by a finite number of point sources:
	\begin{equation}
	\label{min-deltas}
	\min_{(q,u)\in \R^\ell \times H^{1-\sigma}_0}
	\frac12 \snorm{ u-\ud }_0^2 + \frac\alpha2 \sum_{j=1}^{\ell} q_j^2
	\quad\text{subject to}\quad
	-\Delta u= \sum_{j=1}^{\ell} q_j \delta_{x_j},
	\end{equation}
	where the underlying domain $\Omega\subset\R^2$ is planar, polygonal, Lipschitz, but not necessarily convex, $\{ x_j \}_{j=1}^{\ell}\subset \Omega$ are $\ell$ distinct points, $\delta_{x_j}$ denotes the Dirac functional at the point $x_j$, and $0<\sigma<\frac{1}{2}$. The bilinear form $a(v,w)=\int_\Omega \nabla v \cdot \nabla w \, dx$, $v,w \in C^\infty_0(\Omega)$, has a continuous and inf-sup-stable extension on $V_1\times V_2$ with $V_1 = H^{1-\sigma}_0(\Omega)$ and $V_2 = H^{1+\sigma}_0(\Omega)$ and allows for a standard discretization with linear finite elements $S_h$ for both trial and test space; see, e.g., \cite{GaMoVe:2017}. For the verification of the discrete inf-sup condition, denote by $R_h$ and $\Lambda_h$ the Ritz projection and the Scott-Zhang interpolation operator, respectively. As
	\begin{equation*}
	\snorm{R_h\vphi}_{1+\sigma}
	\leq
	\snorm{\Lambda_h\vphi}_{1+\sigma} + \snorm{R_h\vphi-\Lambda_h\vphi}_{1+\sigma}
	\lesssim
	\snorm{\vphi}_{1+\sigma}
	+ h^{-\sigma} \snorm{R_h\vphi-\Lambda_h\vphi}_{1} 
	\end{equation*}
	and
	\begin{align*}
	h^{-\sigma} \snorm{R_h\vphi-\Lambda_h\vphi}_{1}
	&\leq
	h^{-\sigma} \snorm{R_h\vphi-\vphi}_{1}
	+ h^{-\sigma} \snorm{\vphi-\Lambda_h\vphi}_{1}
	\\
	&\lesssim 
	h^{-\sigma} \snorm{\vphi-\Lambda_h\vphi}_{1}
	\lesssim
	\snorm{\vphi}_{1+\sigma},
	\end{align*}
	the continuous inf-sup-condition yields, for any $s_h\in S_h$,
	\begin{equation*}
	\snorm{s_h}_{1-\sigma}
	\lesssim
	\sup_{\snorm{\vphi}_{1+\sigma}=1} a(s_h,\vphi)
	=
	\sup_{\snorm{\vphi}_{1+\sigma}=1} a(s_h,R_h\vphi)
	\lesssim
	\sup_{\vphi_h\in S_h,\norm{\vphi}_{1+\sigma}=1} a(s_h,\vphi_h),
	\end{equation*}
	and so
	\begin{equation*}
	\snorm{s_h}_{1-\sigma}
	\leq
	\mu_h \sup_{\vphi_h\in S_h,\norm{\vphi_h}_2=1} a(s_h,\vphi_h),
	\end{equation*}
	where $\mu_h$ depends only on continuous inf-sup constant and
        on the shape regularity of the underlying mesh and we switched
        to \eqref{adjoint-norm} for the norm on $V_2$.
	To complete the setting, we set $W=L^2(\Omega)$, $Q=\R^\ell$, and let $I$ be the canonical embedding $H^{1-\sigma}(\Omega)\to L^2(\Omega)$ and $C:\R^\ell\to H^{-(1+\sigma)}(\Omega)$ be given by $Cq=\sum_{j=1}^{\ell} q_j \delta_{x_j}$. The continuity constants $M_I$ and $M_C$ are of order $1$ and $\ell$, respectively. Notice that, for $\sigma=0$, $C$ is not continuous because functions in $H^1_0(\Omega)$ do not have point values in general. Choosing $\delta \in (0,\sigma)$, we have \eqref{quantitative} with $s_1=1-\sigma$, $s_2=1+\sigma$ and therefore
	\begin{equation*}
	\nu_h
	=
	\mu_h \left( 1 + \frac{O(h^{\delta})}{\sqrt{\alpha}} \right)
	\quad\text{as}\quad
	h\to0.
	\end{equation*}
\end{ex}

\section{Analysis with approximate control-action operator}
\label{sec:4}
%
%
In this section, we shall analyze the approximation properties of a
variational discretization, where the control-action operator is
approximated. This includes the case of a discretized control space.

\subsection{Approximate variational discretization}
\label{sec:vc-setting}
Let $V_{h,i}\subset V_i$, $i=1,2$, be the same finite-dimensional conforming spaces introduced in Section~\ref{sec:vardisc} and assume that the linear operator $C_h^*:V\to Q$ approximates $C^*$. Then the (semi-)discrete optimization
\begin{equation}
\label{mod-min-Ch}
\begin{aligned}
 \min_{(\Tilde{q}_h,u_h) \in Q\times V_{h,1}}
&\;\frac{1}{2}\norm{Iu_h-\ud}_W^2 + \frac{\alpha}{2}\norm{\Tilde{q}_h}^2_Q
\\
&\text{subject to}
\qquad
\forall \vphi_{h,2} \in V_{h,2}
\quad
a(u_h, \vphi_{h,2})
=
\sprod{\Tilde{q}_h}{C_h^*\vphi_{h,2}}_Q,
\end{aligned}
\end{equation}
generalizes \eqref{mod-min-h}. It has the solution $(\Tilde{q}_h,\Tilde{u}_h) \in Q \times V_{h,1}$ if and only if there exists $\tilde z_h \in V_{h,2}$ such that
\begin{align}
\label{mod-opt-VC}
\begin{aligned}
 &\forall \vphi_{h,2} \in V_{h,2}
 &a(\tilde u_h,\vphi_{h,2})
 &=
 \sprod{\Tilde{q}_h}{C_h^*\vphi_{h,2}}_Q,
\\
 &\forall \vphi_{h,1} \in V_{h,1}
 &a(\vphi_{h,1},\tilde z_h)
 &=
 \tfrac{1}{\sqrt{\alpha}}  \sprod{I\tilde u_h-\ud}{I \vphi_{h,1}}_W, 
\\
&&\sqrt{\alpha} \Tilde{q}_h &= - C^*_h\tilde
z_h. 
\end{aligned}
\end{align}
As before, we may eliminate $\Tilde{q}_h$. If we define
\begin{align*}
 b_h(v,\vphi)
 \defas
 a(v,\vphi) + \frac1{\sqrt{\alpha}}c_h(v,\vphi)
\end{align*}
with
\begin{align*}
 c_h(v,\vphi)
 \defas
 \sprod{C_h^*v_2}{C_h^*\vphi_2}_Q - \sprod{Iv_1}{I\vphi_1}_W
\end{align*}
for $v,\vphi \in V = V_1 \times V_2$, then the reduced version of \eqref{mod-opt-VC} is the following perturbation of the optimality system \eqref{mod-ropt-h}:
\begin{align}
\label{mod-ropt-VC}
\begin{aligned}
\text{find } \tilde x_h&=(\tilde u_h,\tilde z_h) \in V_h \text { such that }
\\
&\forall \vphi_h = (\vphi_{h,1},\vphi_{h,2}) \in V_h
\quad
b_h(\tilde x_h,\vphi_h)
=
- \tfrac{1}{\sqrt{\alpha}}  \sprod{\ud}{I\vphi_{h,1}}_W,
\end{aligned}
\end{align}
where $V_h = V_{h,1}\times V_{h,2}$. Before we proceed to analyze its discretization error, let us give an important class of examples.

\begin{ex}[Discretized controls]
\label{ex:controldisc}
We consider a conforming discretization of the control variable. More precisely, replacing $Q$ in \eqref{mod-min-h} with a finite-dimensional subspace $Q_h \subset Q$ leads to the discrete optimality system
\begin{align}
\label{mod-opt-qh}
\begin{aligned}
	&\forall \vphi_{h,2}\in V_{h,2}
	& a(\tilde{u}_h, \vphi_{h,2})
	&= \sprod{\Tilde{q}_h}{C^*\vphi_{h,2}}_Q,
	\\
	&\forall \vphi_{h,1}\in V_{h,1}
	&a(\vphi_{h,1},\tilde{z}_h)
	&= \tfrac{1}{\sqrt\alpha}
	\sprod{I\tilde{u}_h-\ud}{I \vphi_{h,1}}_W, 
	\\
	&\forall p_h \in Q_h
	&(\sqrt\alpha \Tilde{q}_h,p_h)_Q &= -
	(C^*\Tilde{z}_h,p_h)_Q. 
\end{aligned}
\end{align}
If we denote by $P_h$ the $Q$-orthogonal projection onto $Q_h$, then
the third equations mean 
\[
 \tilde{q}_h = -\tfrac{1}{\sqrt\alpha} P_{h} C^* \Tilde{z}_h
\]
and, therefore, the right-hand side of the first equation can be rewritten as follows:	
\[
 \sprod{\tilde{q}_h}{C^*\vphi_{h,2}}_Q
 =
 -\tfrac{1}{\sqrt\alpha} \sprod{P_{h}C^*\Tilde{z}_h}{C^*\vphi_{h,2}}_Q
 =
 -\tfrac{1}{\sqrt\alpha} \sprod{P_{h}C^*\Tilde{z}_h}{P_{h}C^*\vphi_{h,2}}_Q.
\]
Hence, the reduced version of \eqref{mod-opt-qh} is a special case of \eqref{mod-ropt-VC} with
\begin{equation*}
 C_h^* = P_hC^*.
\end{equation*}
\end{ex}

As the bilinear form $b_h$ coincides with $b$ except for using $C_h^*$ in place of $C$, the non-asymptotic continuity and nondegeneracy properties of $b$ in Section~\ref{sec:2}-\ref{sec:3} immediately carry over by replacing $M_C$ with the operator norm $M_{C_h}$ of $C_h^*$.
In particular, setting $\Tilde{M}_h\defas \max\{M_I,M_{C_h}\}$ and defining 
\begin{equation}
\label{norm-alphaV}
 \norm{\vphi}_{\alpha,h}
 \defas
 M_a\norm{\vphi} + \frac{\Tilde{M}_h}{\sqrt{\alpha}} \snorm{\vphi},
\end{equation}
inequality \eqref{cont} yields
\begin{equation}
\label{contV}
%
 |b_h(v,\vphi)|
 \leq
 M_a \norm{v} \norm{\vphi}
  + \frac{\tilde M_h}{\sqrt{\alpha}} \norm{v} \snorm{\vphi}
 \leq
  \norm{v} \norm{\vphi}_{\alpha,h}
\end{equation}
for all $v,\vphi \in V$. Furthermore, \eqref{infsup-b-Vh} and the inf-sup duality \eqref{inf-sup-duality} for $b_h{}_{|V_h\times V_h}$ imply
\begin{align}\label{infsupVC}
 \sup_{\vphi_h\in V_h,\norm{\vphi_h}_{\alpha,h}=1} b_h(v_h,\vphi_h)
 \ge
 \frac{1}{\tilde\kappa_h\mu_h} \|v_h\|, 
\end{align}
for all $v_h\in V_h$, where
\begin{align}
\label{eq:constsVC}
 \Tilde{\kappa}_h
 =
 \frac{1+2\tilde L}{1+\tilde L} \left(
  1+ \tilde{M}_h \mu_h \left(
   1+\frac{2\tilde M_h}{\sqrt\alpha}
  \right)
 \right)
\quad\text{with}\quad
 \tilde{L}
 =
 \frac{\tilde{M}_h}{\sqrt{\alpha}} 
\end{align}
and $\mu_h$ is the quasi-best-approximation constant of the constraint discretization.

Since the structures of the discrete problems \eqref{mod-ropt-VC} and \eqref{mod-ropt-h} are the same, well-posedness of \eqref{mod-ropt-VC}
follows from Lemma~\ref{L:well-posedness-h}.

\subsection{Approximation}
\label{sec:vc-quasi-best}
As in the error analysis of Section~\ref{S:qb-approx}, we adopt the convenient choice
\begin{equation*}
 \text{\eqref{adjoint-norm} as norm in $V_2$.}
\end{equation*}
Here we start our analysis by splitting the error into an approximation part and a consistency part.

\begin{lem}[Approximation and consistency error]
\label{L:err=approx+consistency}
Let $x=(u,z)$ be any solution of the optimality system
\eqref{mod-vred} and let $\Tilde{x}_h$ 
be its approximation from \eqref{mod-ropt-VC}. Then the error satisfies
\begin{align*}
 \norm{x-\Tilde{x}_h}
 &\le
 \Tilde{\kappa}_h \mu_h \left(
 \inf_{v_h\in V_h} \|x-v_h\|
 + \frac{1}{\sqrt{\alpha}} \sup_{\vphi_h \in V_h}
  \frac{\big\langle (C_h C_h^* - C C^*) z, \vphi_{h,2} \big\rangle_2}{\norm{\vphi_h}_{\alpha,h}}
 \right)
\\
 &\le
  2 \Tilde{\kappa}_h \mu_h \norm{x-\Tilde{x}_h}.
\end{align*}
Here $\tilde{\kappa}_h$ is defined by~\eqref{eq:constsVC}
and $\mu_h$ is the quasi-best-approximation constant of the constraint
discretization from~\eqref{qba-Rh}.
\end{lem}

\begin{proof}
Define $x_h^* \in V_h$ by
\begin{equation*}
 \forall \vphi_h \in V_h
\quad
 b_h(x^*_h,\vphi_h) = b_h(x,\vphi_h).
\end{equation*}
Then Theorem~\ref{T:qba} with $b_h$, $x^*_h$, $\Tilde{\kappa}_h$ in place of $b$, $x_h$, $\kappa_h$ gives
\begin{equation*}
 \norm{x-x^*_h} \leq \Tilde{\kappa}_h \mu_h \inf_{v_h \in V_h} \norm{x-v_h}
\end{equation*}
and we have the identities
\begin{align*}
 b_h(x_h^*-\Tilde{x}_h,\vphi_h)
 &=
 b_h(x-\Tilde{x}_h,\vphi_h)
 =
 b_h(x,\vphi_h) - b(x,\vphi_h)
\\
 &=
 \frac{1}{\sqrt{\alpha}} \langle C_h C_h^*z - C C^* z, \vphi_{h,2} \rangle_2
\end{align*}
for all $\vphi_h \in V_h$.  In view of \eqref{contV} and \eqref{infsupVC}, these identities imply
\begin{equation*}
 \frac{1}{\Tilde{\kappa}_h\mu_h} \norm{x^*_h - \Tilde{x}_h}
 \leq
 \frac{1}{\sqrt{\alpha}}
  \sup_{\vphi_h \in V_h}
   \frac{\big\langle (C_h C_h^* - C C^*) z, \vphi_{h,2} \big\rangle_2}{\norm{\vphi_h}_{\alpha,h}}
 \leq
 \norm{x-\Tilde{x}_h}.
\end{equation*}
The claim follows from the obvious inequalities $\norm{x-\Tilde{x}_h} \leq \norm{x-x^*_h} + \norm{x^*_h-\Tilde{x}_h}$ and $\inf_{v_h \in V_h} \norm{x-v_h} \leq \norm{x-\Tilde{x}_h}$.
\end{proof}

For the next corollary it is necessary to consider a
sufficiently large class of optimization
problems, e.g., the class $\mathcal{P}$ of optimization
problems, where a constraint can be of the form $Au=Cq+f$ for some
$f\in V_2^*$ and $I^*$ may be surjective.  
\begin{cor}[Necessary condition for quasi-best approximation]
\label{C:qba-Ch}
If the approximate variational discretization \eqref{mod-ropt-VC} is quasi-best in the class $\mathcal{P}$, then
\begin{equation*}
 \forall v_{2,h} \in V_{2,h}
\quad
 \norm{C_h^*v_{2,h}}_Q = \norm{C^*v_{2,h}}_Q.
\end{equation*} 
\end{cor}

\begin{proof}
Let $v_{2,h}\in V_{2,h}$ be arbitrary and take some $v_{1,h}\in
V_{1,h}$. Then $v_h=(v_{1,h},v_{2,h}) \in V_h \subset V$ is a possible
solution in the class $\mathcal{P}$. Since \eqref{mod-ropt-VC} is
quasi-best in $\mathcal{P}$, the discrete solution is exactly $v_h \in
V_h$. Hence, by Lemma \ref{L:err=approx+consistency} we have $(C_h
C_h^* - C C^*) v_{2,h} = 0$, which yields $\norm{C_h^*v_{h,2}}_Q = \norm{C^*v_{h,2}}_Q$.
\end{proof}

Although possible, it is difficult to imagine that a practical
approximation $C_h^*$ satisfies the condition in Corollary
\ref{C:qba-Ch} without coinciding with $C$. We therefore consider in
what follows only assumptions on $C_h^*$ that lead to asymptotic
quasi-best approximation. In view of
Lemma~\ref{L:err=approx+consistency}, this requires, that the
consistency error vanishes at least as fast as the best approximation
error, i.e.,
\begin{align}
\label{eq:C(h)}
 \sup_{\vphi_h \in V_h}
  \frac{\big\langle (C_h C_h^* - C C^*) z, \vphi_{h,2}\big\rangle_2}
   {\norm{\vphi_h}_{\alpha,h}}
 =
 o\left(\inf_{v_h\in V_h}\norm{x-v_h}\right).
\end{align}
Moreover, to capture in the limit the compactness of $C^*$ resulting from  assumption~\eqref{compactness}, we assume that  
\begin{align}
\label{compactness-VC}
 d_h\weak 0~\text{weakly in $V_2$ as}~h\to 0
\quad\implies\quad
  C_h^*d_h\to 0 ~\text{strongly as}~h\to 0.
\end{align}
This implies that the operator norms
$\|C_h^*\|_{L(V_2,Q)} = \Tilde{M}_h = \max\{M_I,M_{C_h}\}$
are uniformly bounded. Indeed, suppose that $\Tilde{M}_h \to \infty$ as
$h\to 0$ and, for each $h>0$, let  
$\vphi_2^h\in V_2$ be such that $\|C_h^*\vphi_2^h\|_Q = \Tilde{M}_h$ and
$\|\vphi_2^h\|_2 = 1 $. Then $\vphi_2^h / \Tilde{M}_h\to 0$ in $V_2$ as $h\to
0$, which, in view of~\eqref{compactness-VC}, yields a
contradiction. Consequently,
\begin{equation*}
 \Tilde{M}
 \defas
 \sup_{h}\Tilde{M}_h
 =
 \sup_{h}\max\{M_I,M_{C_h}\}
\end{equation*}
is finite.

\begin{lem}[Qualitative asymptotic quasi-best approximation with approximate control-action]
\label{L:VC-qual-as-qba}
Let $x = (u,z) \in V$ be a solution to problem~\eqref{mod-vred} and let $\tilde{x}_h = (\tilde{u}_h,\tilde{z}_h) \in V_h$, $h>0$, be the corresponding approximations given by \eqref{mod-ropt-VC}. Furthermore, assume uniform stability \eqref{unif-qba}, approximability \eqref{approximability}, limiting compactness~\eqref{compactness-VC}, and that $I:V_1 \to W$ is compact.
If the exact solution $x$ satisfies \eqref{eq:C(h)}, we have 
\begin{align*}
 \norm{x-\Tilde{x}_h}
 \le
 \mu_h \left(1 + \frac{\tilde{\kappa}}{\sqrt{\alpha}} \, o(1)
 \right)
 \inf_{v_h\in V_h}\|x-v_h\|
\quad\text{as}~h\to 0,
\end{align*}
  where
  \begin{align*}
    \Tilde{\kappa}
 =
 \frac{1+2\tilde L}{1+\tilde L} \left(
  1+ \tilde{M} \bar\mu\left(
   1+\frac{2\tilde M}{\sqrt\alpha}
  \right)
 \right)<\infty
\quad\text{with}\quad
 \tilde{L}
 =
 \frac{\tilde{M}}{\sqrt{\alpha}}.
  \end{align*}
\end{lem}

\begin{proof}
As in the proof of Lemma~\ref{L:err=approx+consistency}, define $x_h^* \in V_h$ by
\begin{equation*}
 \forall \vphi_h \in V_h
\quad
 b_h(x^*_h,\vphi_h) = b_h(x,\vphi_h).
\end{equation*}
We deduce
\begin{align}\label{eq:qualxh*}
 \norm{x-x_h^*}
 \le
 \mu_h \big( 1+\tilde \kappa_h\,o(1) \big) \inf_{v_h\in V_h}\norm{x-v_h} 
\end{align}
by replacing $b$ with $b_h$ and $x_h$ with $x_h^*$ in
Lemma~\ref{L:qba-consts} and using the limiting compactness \eqref{compactness-VC} instead of the compactness of $C^*:V_2\to Q$ in the proof of Lemma~\ref{L:aqba-qualitative}.
Next, proceeding as in the proof of
Lemma~\ref{L:err=approx+consistency},
assumption~\eqref{eq:C(h)} on the exact solution gives
\begin{align*}
 \frac{\sqrt{\alpha}}{\tilde\kappa\mu}\norm{x_h^*-\Tilde{x}_h}
 =
 o\left(\inf_{v_h\in V_h}\norm{x-v_h}\right).
\end{align*}
We therefore conclude by inserting the two preceding relationships into the triangle inequality $\norm{x-\Tilde{x}_h}\le \norm{x- x_h^*}+\norm{x_h^*-\Tilde{x}_h}.$
\end{proof}

We turn to prove a quantitative quasi-best approximation result. To this end, we need to specify the qualitative assumptions~\eqref{eq:C(h)} and \eqref{compactness-VC} by quantitative ones. We shall assume that
\begin{align}
\label{eq:C(h)-quant}
 \sup_{\vphi_h \in V_h}
 \frac{\big\langle (C_h C_h^* - C C^*) z, \vphi_{h,2}\big\rangle_2}
 {\norm{\vphi_h}_{\alpha,h}}
 =
 O(h^\delta) \inf_{v_h\in V_h}\norm{x-v_h}
\end{align}
and that
\begin{align}
\label{VC-quant-compactness}
 C_h \in L(Q,H^{-s_2+\delta})\quad\text{is uniformly bounded with
	respect to $h>0$,}
\end{align}
where $\delta>0$ is suitably chosen. Note that \eqref{VC-quant-compactness} reduces for $C_h=C$ to the part regarding $C$ in the quantitative counterpart \eqref{quant-compactness} of the qualitative compactness~\eqref{compactness}.

\begin{thm}[Quantitative asymptotic quasi-best approximation with approximate control-action]
\label{T:VC-quant-as-qba}
Let $x$, $\Tilde{x}_h$, $h>0$, and $\Tilde{\kappa}$ be as in Lemma \ref{L:VC-qual-as-qba}. In addition, assume uniform stability \eqref{unif-qba} and that there exists $\delta>0$ such that we have \eqref{quantitative}, where~\eqref{VC-quant-compactness} replaces the assumption on $C$ in~\eqref{quant-compactness}.
If the exact solution $x$ satisfies also \eqref{eq:C(h)-quant} with the same $\delta$, we have 
\begin{align*}
\norm{x-\Tilde{x}_h}
\le
\mu_h \left(1 + \frac{\tilde{\kappa}}{\sqrt{\alpha}} \, O(h^\delta)
\right)
\inf_{v_h\in V_h}\|x-v_h\|
\quad\text{as}~h\to 0,
\end{align*} 
\end{thm}

\begin{proof}
We follow the lines of the proof of Lemma~\ref{L:VC-qual-as-qba}, but replacing~\eqref{eq:C(h)} with \eqref{eq:C(h)-quant} and \eqref{eq:qualxh*} with a quantitative argument in the spirit of Theorem~\ref{T:aqba-quant}. To this end, it suffices to use~\eqref{VC-quant-compactness} instead of \eqref{quant-compactness}.
\end{proof}

We conclude this section by assessing the key assumptions~\eqref{eq:C(h)} and \eqref{eq:C(h)-quant} by a remark and an example.

\begin{rem}[Ensuring dominated consistency error]
\label{rem:assump-split}
As
\begin{equation*}
 \sup_{\vphi_h \in V_h}
  \frac{\big\langle (C_h C_h^* - C C^*) z, \vphi_{h,2}\big\rangle_2}
   {\norm{\vphi_h}_{\alpha,h}}
 \leq
 \|C_hC_h^*-CC^*\|_{L(V_2,V_{2}^*)}
\end{equation*}
for
\begin{equation*}
 \|C_hC_h^*-CC^*\|_{L(V_2,V_{2}^*)}
 \defas
 \sup_{\vphi_h \in V_h}
  \frac{\big\langle (C_h C_h^* - C C^*) z, \vphi_{h,2}\big\rangle_2}
   {\norm{\vphi_h}_2},
\end{equation*}
we may verify assumptions \eqref{eq:C(h)} and \eqref{eq:C(h)-quant} using  relationships for $\|C_hC_h^*-CC^*\|_{L(V_2,V_{2}^*)}$.
\end{rem}

\begin{ex}[Simple model optimization and piecewise constant controls]
Consider the setting of Example \ref{ex:sim-opt}, but with problem \eqref{sim-min} with linear finite elements for the constraint and piecewise constants for the control variable. In the light of Example \ref{ex:controldisc}, this full discretization can be cast into \eqref{mod-ropt-VC} with $C_h=P_hC$, where $P_h$ is the $L^2$-projection onto piecewise constants. By duality, we have
\begin{equation*}
 \|C_hC_h^*-CC^*\|_{L(V_2,V^*_{2})} \leq c_1 h^2,
\end{equation*}
where $c_1$ depends on the shape regularity of the underlying meshes.
Suppose that there is a constant $c_2$ such that
\begin{equation*}
 \inf_{v_h\in V_h} \norm{x-v_h} \geq c_2 h.
\end{equation*} This holds for example if the matrix norm of the Hessian of the exact state or its adjoint state are bounded away from 0 in a fixed subdomain. We conclude
\begin{equation*}
 \|C_hC_h^*-CC^*\|_{L(V_2,V^*_{2})}
 \leq
 c_1 h^2
 \leq
 \frac{c_1}{c_2} h \inf_{v_h\in V_h} \norm{x-v_h},
\end{equation*}
i.e., \eqref{eq:C(h)-quant} with $\delta=1$ and a constant depending on the  exact solution under consideration.
\end{ex}

\section{Analysis with Control Constraints}
\label{sec:5}
%
%
This section generalize our approach to optimization problems that are
nonlinear because of constraints on the control. 

\subsection{Control constraints and discretization}
%
%
Let $K \subset Q$ be the set of admissible controls. We assume that
\begin{equation}
\label{admissible-controls}
 \text{$K$ is nonempty, closed, and convex}
\end{equation}
and denote by $\Pi_{K}:Q\to K$ the projection operator onto $K$ which
is characterized by $\norm{q-\Pi_{K}q}_Q = \inf_{p \in K}
\norm{q-p}_Q$ or, equivalently, by
\begin{equation*}
 \forall p \in K
\quad
 (q - \Pi_{K}q, \Pi_{K}q - p)_Q
 \ge
 0.
\end{equation*}
The latter characterization implies
\begin{equation}
\label{eq-qmono-lip}
 (\Pi_{K}(q) - \Pi_{K}(p), q-p)_Q
 \ge
 \norm{\Pi_{K}(q)-\Pi_{K}(p)}_Q^2
\end{equation}
for all $q,p \in Q$, which in turn shows that the operator $\Pi_{K}$ is strongly monotone and Lipschitz continuous, in both cases with constant 1.

The generalization of problem~\eqref{mod-min} incorporating
convex control constraints is then the \emph{convex optimization problem}
\begin{equation}
\label{mod-min-cc}
 \min_{(q,u)\in K \times V_1}
  \frac12\norm{ Iu-\ud }_W^2 + \frac\alpha2 \norm{q}_Q^2
\quad\text{subject to}\quad
 Au = Cq.
\end{equation}
Thanks to \eqref{admissible-controls}, a solution $(q,u)$ is characterized by the existence of $z \in V$ such that the following counterpart of the rescaled optimality system \eqref{mod-opt} is satisfied:
\begin{equation}
\label{mod-opt-cc}
 Au = Cq,
\quad
 A^*z = \tfrac{1}{\sqrt\alpha}I^*(Iu-\ud),
\quad
  q = \Pi_{K}(-\tfrac{1}{\sqrt\alpha} C^*z).
\end{equation}
As in Section~\ref{sec:2}, we insert the third equation into the first one and consider the corresponding \emph{weak formulation of the rescaled and reduced optimality system}:
\begin{equation}
\label{mod-ropt-cc}
 \text{find $x\in V$ such that}
\quad \forall \vphi \in V \;\;
  b_K(x,\vphi)
  = -\tfrac{1}{\sqrt\alpha} (\ud,I\vphi_1)_W,
\end{equation}
where $b_K \defas a + c_{K,\alpha}$ and
\begin{equation*}
 c_{K,\alpha}(v,\vphi)
 \defas
 - \left(
  \Pi_{K}(-\tfrac{1}{\sqrt\alpha} C^*v_2), C^*\vphi_2
  \right)_Q
 -\tfrac{1}{\sqrt\alpha}(Iv_1,I\vphi_1)_W,
\end{equation*}
which already incorporates the $1/\sqrt{\alpha}$-scaling. In contrast to the previous sections, $c_{K,\alpha}$ and so $b_K$ are in general not linear in the first argument. Nonetheless, if we introduce the pseudometric
\begin{equation*}
 \delta_{K,\alpha}(v,w)^2
 \defas
 \alpha \norm{\Pi_{K}(-\tfrac{1}{\sqrt{\alpha}}C^*v_2) - \Pi_{K}(-\tfrac{1}{\sqrt{\alpha}}C^*w_2) }_Q^2
 +
 \norm{I(v_1-w_1)}_Q^2,
\end{equation*}
inequality \eqref{eq-qmono-lip} leads to the following replacement of the properties \eqref{cont-inf-sup-c} of the bilinear form $c$: if $v,w\in V$ and $\vphi = \big({-}(v_1-w_1), v_2 - w_2 \big)$, then
\begin{subequations}
\label{cont-coercivity-c-cc}
\begin{equation}
 c_{K,\alpha}(v,\vphi) - c_{K,\alpha}(w,\vphi)
 \geq
 \frac{1}{\sqrt{\alpha}} \delta_{K,\alpha}(v, w)^2,
\end{equation}
while, for any $v,w,\vphi \in V$ arbitrary, we have, 
\begin{equation}
\label{cont-c-cc}
 |c_{K,\alpha}(v,\vphi) - c_{K,\alpha}(w, \vphi)|
 \leq
 \frac{1}{\sqrt{\alpha}} \delta_{K,\alpha}(v,w) \snorm{\vphi}.
\end{equation}
\end{subequations}
In addition, we have, for $v, w \in V$,
\begin{equation}
\label{delta<enorm}
\delta_{K,\alpha}(v,w)
\leq
\snorm{v-w}.
\end{equation}
The continuity bound \eqref{cont-c-cc} leads to
\begin{equation}
\label{cont-cc}
 |b_K(v,\vphi) - b_K(w,\vphi)|
 \leq
 d_{K,\alpha}(v,w) \norm{\vphi}
\end{equation}
with the metric
\begin{equation*}
 d_{K,\alpha}(v,w)
 \defas
 M_a \norm{v-w} + \frac{M}{\sqrt{\alpha}}\delta_{K,\alpha}(v,w),
\quad
 v,w \in V.
\end{equation*}

Notice that the role of the two arguments of $c$ and $b_K$ cannot be interchanged. We adapt \eqref{inf-sup-ansatz} to this new situation in the following way: given $v,w \in V$, we choose $\vphi = T_K (v-w)$, where $T_K:V\to V$ is the linear operator given by 
\begin{equation}
\label{T_K}
 T_K\psi
 \defas
 m_a(A^{-1}J_2\psi_2, A^{-*}J_1\psi_1) + \gamma (-\psi_1,\psi_2),
\end{equation}
$\gamma$ as in \eqref{gamma}, and $J_i:V_i \to V_i^*$ is the Riesz map for $V_i$, $i=1,2$. In view of \eqref{inf-sup-const2}, we thus obtain the following counterpart of Theorem \ref{T:cont-inf-sup}.

\begin{thm}[Properties of form $b_K$]
\label{T:bK}
If we equip $V$ as trial space with $d_{K,\alpha}$ and as test space with $\norm{\cdot}$, then we have, for any $v,w,\vphi \in V$,
\begin{align*}
 b_K \big( v, T_K(v-w) \big) - b_K \big( w,T_K(v-w) \big)
 &\geq
 \frac{1+L}{1+2L} \frac{m_a}{M_a}\,d_{K,\alpha}(v,w) \norm{v-w}
\\
 &\geq
 \frac{1}{\kappa} \frac{m_a}{M_a}\,d_{K,\alpha}(v,w) \norm{T_K(v-w)}
\end{align*}
and
\begin{equation*}
 | b_K(v,\vphi) - b_K(w,\vphi) |
 \leq
 d_{K,\alpha}(v,w) \norm{\vphi},
\end{equation*}
where $\kappa$ is defined by \eqref{def-kappa}.
\end{thm}

Also here, we can conclude existence and uniqueness as a side-product.
\begin{cor}[Well-posedness with control constraints]
\label{C:uniqueness-cc} 
The optimization problem~\eqref{mod-ropt-cc} has a unique solution.
\end{cor}
\begin{proof}
We shall apply the Zarantonello's theorem of strongly monotone operators
\cite[Theorem 25.B]{ZeidlerIIB} in the Hilbert space $(V,\norm{\cdot})$.
To prepare this, we first observe that
\begin{equation}
\label{T_K-isomorfism}
 \text{$T_K$ is a linear isomorphism on $(V,\norm{\cdot})$.}
\end{equation}
Indeed, it is continuous with constant $1+\gamma$ owing to \eqref{inf-sup-ansatz;norm-test-fct} and boundedly invertible on account of the consequence
\begin{equation*}
 \frac{1+L}{1+2L} \frac{m_a}{M_a} \norm{v} \norm{v}_\alpha
 \leq
 b(T_K v, v)
 \leq
 \norm{T_K v} \norm{v}_{\alpha} 
\end{equation*}
of \eqref{cont} and \eqref{inf-sup-const2} for the bilinear form $b$.
Let us consider the nonlinear operator $\widetilde{B}_K : V\to V^*$ defined by 
\begin{align*}
 \langle \widetilde{B}_K v,\,\vphi \rangle
 =
 b_K(v,T_K \vphi),
\end{align*}
where $\langle\cdot,\cdot\rangle$ denotes the duality pairing associated with $(V,\norm{\cdot})$. Making use of Theorem~\ref{T:bK}, \eqref{cont} and \eqref{delta<enorm}, we see that, for all $v,w\in V$,
\begin{equation*}
 \langle \widetilde{B}_K v - \widetilde{B}_K w, v-w \rangle
 \geq
 \frac{1+L}{1+2L} m_a \norm{v-w}^2
\end{equation*}
and
\begin{equation*}
 \langle \widetilde{B}_K v - \widetilde{B}_K w, \vphi \rangle
 \leq
 \left( 
  M_a + \frac{M^2}{\sqrt{\alpha}}
 \right) (1+\gamma) \norm{v-w} \norm{\vphi}.
\end{equation*}
Hence, $\widetilde{B}_K$ is strongly monotone and Lipschitz continuous
and therefore boundedly invertible by \cite[Theorem
25.B]{ZeidlerIIB}. In light of~\eqref{T_K-isomorfism}, we can conclude by noting $T_K^{-*}\widetilde{B}_Kv = b_K(v,\cdot)$ for all $v\in V$.
\end{proof}

In order to discretize the optimization problem \eqref{mod-min-cc} with control constraints, we proceed as in Section~\ref{sec:vardisc}. Introducing the discrete space $V_h = V_{h,1} \times V_{h,2}$ as therein, the variational discretization can be characterized as follows:
\begin{equation}
\label{mod-ropt-cc-h}
 \text{find } x_h \in V_h \text{ such that}
\quad
 \forall \vphi_h \in V_h \;\; b_K(x_h,\vphi_h)
 =
 - \frac{1}{\sqrt{\alpha}} (\ud,I\vphi_{h,1})_W. 
\end{equation}
Here we need that $\Pi_{K}(-C^*v_{h,2}/\sqrt{\alpha})$ can be
evaluated exactly for $v_{h,2} \in V_{h,2}$. This occurs, for example,
when we consider \eqref{sim-min} with box constraints and discretize
with linear finite elements. If $\Pi_{K}$ has to be approximated, the
subsequent error analysis involves additional technicalities, similar
to those addressed in Section~\ref{sec:4}. 

Existence and uniqueness of solutions to \eqref{mod-ropt-cc-h} can be
established in a similar way as Corollary \ref{C:uniqueness-cc}. Using
\eqref{adjoint-norm} as in norm in $V_2$, the major change is to
replace the operator \eqref{T_K} by $T_{K,h}: V_h \to V_h$ given by 
\begin{equation}
 T_{K,h}\psi_h
 \defas
 \frac{1}{\mu_h} (A_h^{-1}J_{h,2}\psi_2, A_h^{-*}J_{h,1}\psi_{h,1})
  + \gamma (-\psi_{h,1},\psi_{h,2}),
\end{equation}
where $A_h v_{h,1} \defas a(v_{h,1},\cdot)_{|V_{h,2}}$, $v_{h,1} \in V_{h,1}$, is the discrete counterpart of $A$, $1/\mu_h$ is its inf-sup-constant, $\gamma$ is as in \eqref{def-kappa}, and $J_{h,i}:V_{h,i} \to V_{h,i}^*$ is the Riesz map for $V_{h,i}$, $i=1,2$.

\subsection{Quasi-best approximation}
%
We analyze the quasi-best-approximation properties of the nonlinear variational discretization \eqref{mod-ropt-cc-h}, adopting again
\begin{equation*}
 \text{\eqref{adjoint-norm} as norm in $V_2$.}
\end{equation*}

The following non-asymptotic result draws heavily on Theorem~\ref{T:bK}, which needed  an $\alpha$-dependent error notion for $V$ as trial space.

\begin{thm}[Non-asymptotic quasi-best approximation with control constraints]
\label{T:qba-cc}
If $x_h$ is the approximation given by \eqref{mod-ropt-cc-h} to an arbitrary solution $x$ of \eqref{mod-ropt-cc}, then its error is quasi-best in $V_h$ in that
\begin{equation*}
  d_{K,\alpha}(x,x_h)
  \leq
  (\kappa_h \mu_h+1) \inf_{v_h \in V_h} d_{K,\alpha} (x, v_h),
\end{equation*}
where $\kappa_h$ and $\mu_h$ are as in Theorem \ref{T:qba}.
\end{thm}

\begin{proof}
Given any $v_h \in V_h$, we first write
\begin{equation}
\label{bad-triangle-ineq}
  d_{K,\alpha}(x,x_h)
  \leq
  d_{K,\alpha}(x,v_h) + d_{K,\alpha}(v_h,x_h).
\end{equation}
To bound the second term, we employ Theorem \ref{T:bK} with, respectively, $V_h$, $T_{K,h}$, $1/\mu_h$, $1$, and $\kappa_h$ in place of $V$, $T_K$, $m_a$, $M_a$, and $\kappa$. Writing $\vphi_h = T_{K,h}(v_h-x_h)$, the definitions of $x$ and $x_h$ thus yield, 
\begin{align*}
 \frac{1}{\kappa_h \mu_h} d_{K,\alpha}(v_h,x_h) \norm{\vphi_h}
 &\leq
 b_K(v_h,\vphi_h) - b_K(x_h,\vphi_h)
\\
 &=
 b_K(v_h,\vphi_h) - b_K(x,\vphi_h)
 \leq
 d_{K,\alpha}(x,v_h) \norm{\vphi_h}
\end{align*} 
and the claimed inequality is established as $T_{K,h}$ is invertible.
\end{proof}

The ``$+1$'' in the bound for the quasi-best-approximation constant in Theorem~\ref{T:qba-cc} arises from the triangle inequality \eqref{bad-triangle-ineq}, which is avoided in deriving in \eqref{qba-const}.
Yet, the following asymptotic quasi-best approximation results involving the generalized Ritz projection from \eqref{qba-Rh} are not affected by such an augmentation.

\begin{lem}[Nonlinear variational and generalized Ritz approximations]
\label{L:var-constraint-disc}
Let $x$ and $x_h$ be as in Theorem~\ref{T:qba-cc}. The generalized
Ritz projection $R_hx$ of $x$ and $x_h$ are related by 
\begin{equation*}
 d_{K,\alpha}(x_h,R_hx)
 \leq
 \kappa_h \mu_h \frac{M}{\sqrt{\alpha}} \snorm{x-R_hx},
\end{equation*}
where $\kappa_h$ and $\mu_h$ are as in Theorem~\ref{T:qba}.
\end{lem}

\begin{proof}
Applying Theorem \ref{T:bK} with the setting as in Theorem~\ref{T:qba-cc}, writing $\vphi_h = T_{K,h}(x_h-R_hx)$, and recalling \eqref{delta<enorm}, we derive
\begin{align*}
 \frac{1}{\kappa_h\mu_h} d_{K,\alpha}(x_h,R_h x) \norm{\vphi_h}
 &\leq
 b_K(x_h,\vphi_h) - b_K(R_hx,\vphi_h)
\\
 &=
 b_K(x,\vphi_h) - b_K(R_hx,\vphi_h)
\\
 &=
 c_{K,\alpha}(x,\vphi_h) - c_{K,\alpha}(R_hx,\vphi_h)
 \leq
 \frac{M}{\sqrt{\alpha}} \snorm{x-R_hx} \norm{\vphi_h}
\end{align*}
and, again thanks to the invertibility of $T_{K,h}$, the proof is finished.
\end{proof}

Let us sharpen Lemma~\ref{L:var-constraint-disc} with the help of the additional assumptions and arguments from Section~\ref{S:asym-qba} regarding the linear optimality system. 

\begin{thm}[Supercloseness to the generalized Ritz approximation]
\label{T:supercloseness}
Let $x$, $x_h$, and $R_hx$ be as in Lemma \ref{L:var-constraint-disc}.
Moreover, assume \eqref{unif-qba} and define $\bar{\kappa}$ as in Lemma~\ref{L:aqba-qualitative}.
If \eqref{qualitative} holds, then
\begin{equation*}
 d_{K,\alpha}(x_h,R_hx)
 \leq
 \frac{M}{\sqrt{\alpha}} \bar{\kappa} \bar{\mu} \, o(\norm{x-R_hx})
\text{ as }h\to0.
\end{equation*}
More specifically, if \eqref{quantitative} holds, then
\begin{equation*}
 d_{K,\alpha}(x_h,R_hx)
 \leq
 \frac{M}{\sqrt{\alpha}} \bar{\kappa} \bar{\mu} \, O(h^\delta\norm{x-R_hx})
\text{ as }h\to0.
\end{equation*}
For the $\alpha$-dependence of $\bar{\kappa}$, cf.\ Remark~\ref{R:vanish-regularization}.
\end{thm}

\begin{proof}
In view of Lemma \ref{L:var-constraint-disc}, it suffices to show $\snorm{x-R_hx}=o(\norm{x-R_hx})$. To this end, we modify the argument in Lemma~\ref{L:aqba-qualitative} slightly; a similar argument has been used by~\cite{FeiFuehPraet:2014} under weaker assumptions on $(V_h)_h$.
Let $(h_k)_k$ be any sequence with $\lim_{k\to\infty} h_k = 0$ and, writing $k$ whenever $h_k$ is an index, consider
\begin{equation*}
 d_k
 \defas
 \begin{cases}
  \displaystyle \frac{x-R_kx}{\norm{x-R_kx}}, &\text{if }x\neq R_kx, \\
  0, &\text{otherwise.}
 \end{cases}
\end{equation*}
The sequence $(d_k)_k$ is bounded in the Hilbert space $V$ by definition. For its weak limit $d\in V$, we have
\begin{equation*}
 a(d,\vphi)
 =
 a(d-d_k,\vphi)
 +
 a(d_k, \vphi - \vphi_k) 
\end{equation*}
for arbitrary $\vphi \in V$ and $\vphi_k \in V_h$. Consequently, \eqref{approximability}, $k\to\infty$, and \eqref{cont-inf-sup-vec-a} yield $d=0$. In view of \eqref{compactness}, $d_k \to 0$ weakly in $V$ then implies $\snorm{d_k} \to 0$.

For the second statement, we just note that the main step of the proof of Theorem~\ref{T:aqba-quant} with $v=x-R_hx$ leads to $\snorm{v-R_hv}=O(h^\delta\norm{x-R_hx})$.
\end{proof}



In view of the inverse triangle inequality 
\begin{equation*}
 \big| \norm{x-x_h} - \norm{x-R_hx} \big|
 \leq
 \norm{x_h-R_hx}
 \leq
 d_{K,\alpha}(x_h,R_h x),
\end{equation*}
Theorem \ref{T:supercloseness} readily yields the following asymptotic quasi-best approximation result.

\begin{cor}[Asymptotic quasi-best approximation with control constraints]
\label{C:aqba-cc}
Let $\nu_{K,h}$ be the quasi-best-approximation constant for the nonlinear variational discretization \eqref{mod-ropt-cc-h} with respect to $\norm{\cdot}$.
Moreover, assume \eqref{unif-qba} and define $\bar{\kappa}$ as in Lemma~\ref{L:aqba-qualitative}. If \eqref{qualitative} holds, then
\begin{equation*}
 \nu_{K,h}
 \leq
 \mu_h \left( 1 + \frac{M}{\sqrt{\alpha}} \bar{\kappa} \, o(1) \right)
\text{ as }h\to0.
\end{equation*}
More specifically, if \eqref{quantitative} holds, then
\begin{equation*}
 \nu_{K,h}
 \leq
 \mu_h \left( 1 + \frac{M}{\sqrt{\alpha}} \bar{\kappa} \, O(h^\delta) \right)
\text{ as }h\to0.
\end{equation*}
For the $\alpha$-dependence of $\bar{\kappa}$, cf.\ Remark~\ref{R:vanish-regularization}.
\end{cor}

In comparison with Lemma~\ref{L:aqba-qualitative} and
Theorem~\ref{T:aqba-quant}, Corollary \ref{C:aqba-cc} features an
additional $M/\sqrt{\alpha}$-factor. This factor stems from the fact
that the derivation we went through used an error notion that also incorporates it. 

%
%
%
%

%
%
%